\documentclass{article}

\usepackage{enumerate}
\usepackage{indentfirst}
\usepackage{amsmath}
\usepackage{amssymb}
\usepackage{amsthm}
\usepackage[mathscr]{euscript}
\usepackage{mathtools}
\usepackage{relsize}
\usepackage{tikz-cd}
\usepackage{macros}
\usepackage{url}
\usepackage{tikz}
\usetikzlibrary{arrows,chains,matrix,positioning,scopes}

\title{Mapping class groups of surfaces with noncompact boundary components}
\author{Ryan Dickmann}
\date{\today}

\begin{document}

\maketitle

\begin{abstract}
We show that the pure mapping class group is uniformly perfect for a certain class of infinite type surfaces with noncompact boundary components. We then combine this result with recent work in the remaining cases to give a complete classification of the perfect and uniformly perfect pure mapping class groups for infinite type surfaces. We also develop a method to cut a general surface into simpler surfaces and extend some mapping class group results to the general case.
\end{abstract}

\tableofcontents

\section{Introduction} \label{intro}
Let $S$ be a connected, orientable, and second-countable surface, possibly with boundary. The \textit{mapping class group} $\MCG(S)$ is the group of all isotopy classes relative to the boundary of $S$ of orientation preserving homeomorphisms of $S$. The elements in this group are considered up to isotopy relative to the boundary. A finite type surface refers to a surface with $\pi_1(S)$ finitely generated, and otherwise we say a surface is infinite type. The $\MCG(S)$ for infinite type surfaces are commonly referred to as \textit{big mapping class groups}. These groups have been the recent focus of many papers, but the case of noncompact boundary components has been largely untouched with only a single paper known to the author considering such groups \cite{fabel}.

The \textit{pure mapping class group} $\PMCG(S)$ is the subgroup of $\MCG(S)$ consisting of elements which fix the ends of $S$, and $\PMCGc(S)$ is the subgroup of compactly supported elements. We equip these groups with the natural compact-open topology. Recently George Domat (and the author in one case) has shown the following.

\begin{theorem} \cite{domat2020big}
Let $S$ be any infinite type surface with only compact boundary components. Then $\PMCGcc{S}$ and $\PMCG(S)$ are not perfect.
\end{theorem}

This partially answered Problem 8 from \cite{APV2017}. In the finite type case, it is a well-known result of Powell that pure mapping class groups are perfect for genus at least 3 \cite{Pow1978}. Surprisingly, a new phenomenon occurs when we also consider surfaces with noncompact boundary components, and even though the general case seems extremely complicated at first glance, it turns out it is possible to completely classify the surfaces with perfect or uniformly perfect pure mapping class groups. A \textit{Disk with Handles} will refer to a surface which can be constructed from a disk by removing points from the boundary and then attaching infinitely many handles accumulating to some subset of these points. We say compact boundary components are added to a surface when we delete open balls with disjoint closures from the interior. We say punctures are added when we remove isolated interior points.

\begin{theorem_a} \label{main} 
Let $S$ be an infinite type surface. Then

\begin{itemize}
    \item $\overline{\PMCG_c(S)}$ is uniformly perfect if and only if $S$ is a Disk with Handles.
    \item $\overline{\PMCG_c(S)}$ is perfect if and only if $S$ is a connected sum of finitely many Disks with Handles with possibly finitely many punctures or compact boundary components added.
\end{itemize}

\end{theorem_a}

In \cite{APV2017}, it was shown for surfaces with only compact boundary components that $\PMCG(S) = \PMCGcc{S}$ if and only if $S$ has at most one end accumulated by genus, and otherwise $\PMCG(S)$ factors as a semidirect product of $\PMCGcc{S}$ with some $\Z^n$ where $n$ is possibly infinite. See Theorem \ref{structure} for the precise statement of this theorem. Once we extend this result to the general case, we immediately get a classification of the perfect $\PMCG(S)$. A Disk with Handles with exactly one end will be referred to as a \textit{Sliced Loch Ness Monster}.\footnote{This name was chosen because the interior of such a surface is often referred to as the \textit{Loch Ness Monster}. The author apologizes for adding to the already out of hand terminology.} Roughly speaking, a degenerate end refers to a end which is the result of deleting an embedded closed subset of the Cantor set from the boundary of a surface (see Definition \ref{degenerate}). For the following theorem, we throw out surfaces with degenerate ends to give a classification which better fits the chosen definition of a Sliced Loch Ness Monster.

\begin{theorem_b} \label{upgradedmain}
Let $S$ be an infinite type surface without degenerate ends. Then

\begin{itemize}
    \item $\PMCG(S)$ is uniformly perfect if and only if $S$ is a Sliced Loch Ness Monster. 
    \item $\PMCG(S)$ is perfect if and only if $S$ is a Sliced Loch Ness Monster with possibly finitely many punctures or compact boundary components added.
\end{itemize}
\end{theorem_b}

Since a Sliced Loch Ness Monster has a single end, the pure mapping class group and the mapping class group coincide. Therefore, this also gives new examples of surfaces with uniformly perfect mapping class groups. These results show there is an interesting distinction between these mapping class groups and the previously studied cases. In particular, the results of Powell and Domat demonstrate a consistent behavior for pure mapping class groups of surfaces without noncompact boundary components, but the cases we study demonstrate a more complicated behavior. Also many of the tools from the other cases do not easily extend as one would hope, so new techniques need to be discovered. 

Disks with Handles and Sliced Loch Ness Monsters will be an essential part of this paper. In Section 4 we will show how to cut a Disk with Handles into a collection of Sliced Loch Ness Monsters, so we can use these simpler surfaces as building blocks for a general argument. We can summarize the decomposition results with the following theorem which is partially inspired by a result in \cite{APV2017}. See Section \ref{background} for some of the terminology.

\begin{theorem_c} \label{fulldecomp}
Every Disk with Handles without planar ends can be cut along a collection of disjoint essential arcs into Sliced Loch Ness Monsters.

Furthermore, any infinite type surface with infinite genus and no planar ends can be cut along disjoint essential simple closed curves into components which are either            
\begin{enumerate}[(i)]
    \item Loch Ness Monsters with $k \in \N \cup \{\infty\}$ compact boundary components added possibly accumulating to the single end.\footnote{Here we are using $\N = \{0, 1, 2, ...\}.$}
    \item Disks with Handles with $k \in \N \cup \{\infty\}$ compact boundary components added possibly accumulating to some subset of the ends. 
\end{enumerate}
\end{theorem_c}

\subsection*{Acknowledgements} The author is indebted to his advisor, Mladen Bestvina, for helping him throughout this project. Thanks to George Domat for many helpful conversations and for writing up the first version of the fragmentation lemma which is critical to this paper. Thanks to Dan Margalit for introducing the author to mapping class groups. Thanks to the referee for suggesting numerous improvements and corrections.

\section{Outline} \label{outline}

In Section \ref{background} we discuss the necessary background including the Classification of Surfaces for orientable noncompact surfaces. The case of compact boundary was done by Kerékjártó \cite{Kerekjarto1923} and Richards \cite{Richards1963}. The general case was done by Brown and Messer \cite{classification}. We also give examples of surfaces which demonstrate the interesting new phenomena that occur for surfaces with noncompact boundary. Some understanding of the general classification and the possible cases may be useful to the usual infinite type surface researcher especially when considering arguments involving cutting a surface along noncompact objects such as a union of infinitely many curves or a union of lines or rays. 

In Section \ref{decomp} we prove Theorem C, and also define the \textit{boundary chains} of a surface with noncompact boundary components (see Definition 4.2). Intuitively speaking, a boundary chain can be thought of as a collection of noncompact boundary components which can be realized in the surface as a circle with points removed. 

In Section \ref{mainresults} we prove Theorem A. The proof that $\PMCGcc{S}$ is uniformly perfect for a Disk with Handles uses standard tricks for writing elements as commutators (see for example the proof that the symmetric group on a countably infinite set is uniformly perfect \cite{ore}). First we use a fragmentation lemma (see Lemma \ref{fraglemma}) to decompose a map in $\PMCGcc{S}$ into a product of two simpler maps. Then after decomposing the surface into simpler pieces using Theorem C, we can apply a standard trick to write each of the simpler maps as a single commutator. 
\begin{figure}[ht]
\begin{center}
\resizebox{.8\textwidth}{!}{
\begin{tikzpicture}
\node[anchor=south west,inner sep=0] at (0,0) {\includegraphics[scale = 0.7]{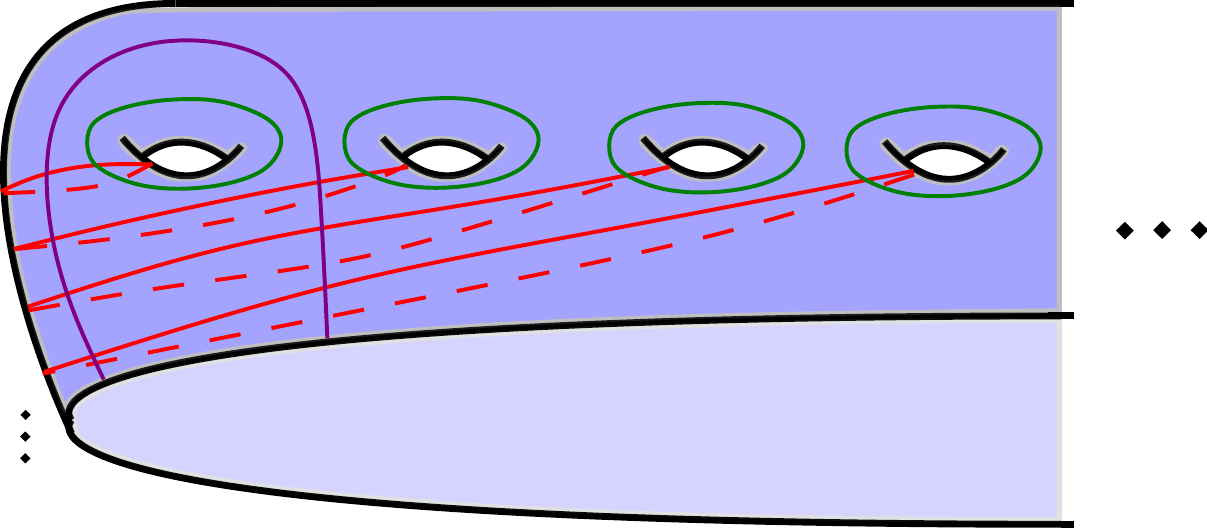}};
\node[anchor=south west,inner sep=0] at (-.35,2.3) {\tiny $\alpha_1$};
\node[anchor=south west,inner sep=0] at (-.3,1.85) {\tiny $\alpha_2$};
\node[anchor=south west,inner sep=0] at (-.2,1.45) {\tiny $\alpha_3$};
\node[anchor=south west,inner sep=0] at (-0.05,1.05) {\tiny $\alpha_4$};

\node[anchor=south west,inner sep=0] at (.6,3) {\tiny $\beta_1$};
\node[anchor=south west,inner sep=0] at (2.4,3) {\tiny $\beta_2$};
\node[anchor=south west,inner sep=0] at (4.3,3) {\tiny $\beta_3$};
\node[anchor=south west,inner sep=0] at (6,3) {\tiny $\beta_4$};

\node[anchor=south west,inner sep=0] at (2,3.3) {\smaller $\gamma$};

\end{tikzpicture}
}

\end{center}
\caption{A Sliced Loch Ness Monster with two infinite collections of curves. The collection $\{\beta_i\}$ eventually leaves every compact subsurface, but every curve in the collection $\{\alpha_i\}$ intersects an arc $\gamma$.}
\label{slicedloch}
\end{figure}
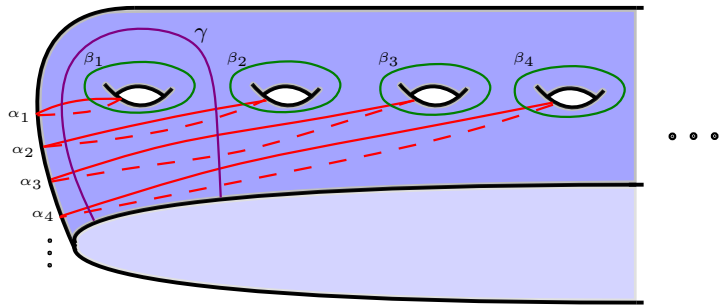

In Section \ref{APV} we discuss how to extend the work of \cite{APV2017} to the general case (see Theorem \ref{structureextended}) and then prove Theorem B. The main proof in Section 6 involves a natural way to turn a surface with noncompact boundary components into one without them via capping the boundary chains (see Construction \ref{capping}). We first extend the Alexander method to the general case (see Theorem 6.4) using a doubling trick. We also extend some well-known facts to the general case (see Lemma 6.2 and Theorem 6.8).

One natural question that first comes to mind is whether the mapping class groups of surfaces with noncompact boundary are even different at all from the compact boundary counterparts. Is every one of these mapping class groups just naturally isomorphic to some mapping class group for a surface with (possibly empty) compact boundary? To the contrary, the following example shows that the mapping class group for a surface with noncompact boundary can correspond to a proper subgroup of the mapping class group for the interior surface. Consider the surface with infinite genus, one end, one noncompact boundary component, and no compact boundary components. It follows from the Classification of Surfaces in Section \ref{background} that there is a unique surface with these properties. This is the \textit{1-Sliced Loch Ness Monster}, and we denote it by $L_s$. If we take an infinite collection of curves $\{\alpha_i\}$ accumulating to the boundary as in Figure \ref{slicedloch}, then the infinite product of Dehn twists $\cdots  T_{\alpha_3}T_{\alpha_2}T_{\alpha_1}$ does not correspond to a homeomorphism of $L_s$. To see this take another infinite collection of curves $\{\beta_i\}$ and an arc $\gamma$ as shown in the figure. If we let $L$ be the interior of $L_s$, then the infinite product of twists corresponds to a well-defined homeomorphism $T = \Pi_{i=1}^{\infty}T_{\alpha_i} \in \MCG(L)$. Restricting maps on $L_s$ to the interior induces a homomorphism $$i: \MCG(L_s) \rightarrow \MCG(L)$$ but $T$ is not in the image. Assume otherwise, and conflate $T$ with a homeomorphism on $L_s$ which restricts to $T$ on $L$. Note $T(\gamma)$ intersects all of the $\beta_i$, so it follows the image is not compact, a contradiction. This follows a similar argument from Proposition 7.1 in \cite{PV2018}. We extend this type of argument to a more general setting in Theorem 6.9.

It will follow from Lemma \ref{injective} that $i$ is injective. Since we have just shown that $i$ is not surjective, we see that $\MCG(L_s)$ truly corresponds to a proper subgroup of $\MCG(L)$. Note more work must be done to show that $\MCG(L_s)$ and $\MCG(L)$ are not abstractly isomorphic. Once we are done though, this will follow from Theorem A. 

The above example also partially motivated some of this work. In \cite{domat2020big}, Domat shows that certain multitwists (a product of powers of Dehn twists about disjoint curves) cannot be written as a product of commutators in $\PMCGcc{S}$. These multitwists involve a collection of curves similar to the $\alpha_i$ in Figure 1. The hope was that a natural subgroup without these types of multitwists would be a perfect group.
\section{Background} \label{background}

\subsection{Classification of Surfaces}

\subsubsection{Compact Boundary}

Here we summarize the classification theorems from  \cite{Richards1963} and \cite{classification} starting with the case of compact boundary. We briefly review the necessary terminology. We always let a surface refer to a connected, orientable, and second-countable 2-manifold. We will assume subsurfaces are connected unless stated otherwise. A \textit{complementary domain} of a surface $S$ is a subsurface which is the closure of some component of $S \setminus K$ for a compact subsurface $K$. 

\begin{definition} \label{exiting}
An \textit{exiting sequence} for a surface $S$ is a sequence of subsurfaces $\{U_{i}\}$ such that the following properties hold:

\begin{itemize} 
    \item $U_{i+1} \subset U_i$ for all $i$.
    \item $\bigcap_{i=1}^{\infty} U_{i} = \emptyset$.
    \item Each $U_i$ is a complementary domain.
\end{itemize}    
\end{definition}

Two such sequences $\{U_{i}\}$ and $\{U_{i}'\}$ are considered equivalent if for any $i$ there exists a $j$ with $U_{j} \subset U_{i}'$, and vice versa. This defines an equivalence relation on the set of exiting sequences, and an equivalence class is referred to as an \textit{end} of the surface. The \textit{ends space} of $S$ is the collection of all equivalence classes, and it is denoted by  $\operatorname{E}(S)$. Note that for a given compact exhaustion the complementary domains of the compact subsurfaces can be used to build exiting sequences. The ends space is an invariant which does not depend on the choice of a compact exhaustion here.

For a given subsurface $U$, let $U^\star$ be the set of ends such that there is a representative sequence eventually contained within $U$. We now equip $E(S)$ with a basis generated by sets of the form $U^\star$ ranging over all subsurfaces $U$ such that $U$ is a complementary domain. This basis gives a topology on $E(S)$ which is totally disconnected, second-countable, compact, and Hausdorff (see \cite{ahlfors} for a standard reference). Topological spaces with these properties are always homeomorphic to a closed subset of the Cantor set.

We say an end is \textit{accumulated by genus} if there is a representative sequence $\{U_i\}$ so every $U_{i}$ has infinite genus. We denote the set of ends accumulated by genus by $\operatorname{E}_{\infty}(S)$. An end is \textit{planar} if there is a representative sequence so some $U_{i}$ is homeomorphic to a subset of the plane. The space of planar ends is exactly $E(S)\setminus E_{\infty}(S)$. We say an end is \textit{isolated} if it is isolated in the topology on the space of ends. Isolated planar ends are referred to as \textit{punctures}.

When we consider surfaces with compact boundary, there is the following classification theorem.

\begin{theorem} [Classification of Surfaces with compact boundary, \cite{Richards1963}] \label{partialclassificiation}
   Two surfaces with compact boundary are homeomorphic if and only if they have homeomorphic pairs $(E(S), E_{\infty}(S))$, the same genus, and the same number of compact boundary components. 
\end{theorem}

\subsubsection{Noncompact Boundary}

Now we summarize the ideas for the general case following \cite{classification}. The previous definitions all apply to a general surface without adaptation, but we need more information to capture all the new possibilities. Note compact or more generally finite type exhaustions for a surface $S$ with noncompact boundary components must include subsurfaces whose boundary intersects the noncompact boundary components of $S$ in a union of intervals. 

For a surface with infinitely many compact boundary components, we must record the ends which are accumulated by these components. We refer to these as \textit{ends accumulated by compact boundary}, and we denote the space of these ends by $E_\partial(S)$. This can be precisely defined in a similar manner to accumulated by genus.

Let $\hat{\partial}S$ be the disjoint union of the noncompact boundary components of a surface $S$. Let $E(\hat{\partial}S)$ be the set of ends of $\hat{\partial}S$. This is just a discrete space with two points associated to each component. Let $v: E(\hat{\partial}S) \to E(S)$ be the function that takes an end of a noncompact boundary component to the end of the surface to which it corresponds. Note it is possible that both ends of a noncompact boundary component get mapped by $v$ to the same end of $S$, as is the case for the 1-Sliced Loch Ness Monster from Figure 1. 

Let $e: E(\hat{\partial}S) \to \pi_0(\hat{\partial}S)$ be the map that takes an end to the corresponding noncompact boundary component. Here $\pi_0(\hat{\partial}S)$ denotes the discrete set of noncompact boundary components of $S$. If we fix an orientation on $S$, then for an arbitrary component $p \in \pi_0(\hat{\partial}S)$ we may distinguish the right and left ends of $e^{-1}(p)$. An \textit{orientation} of $E(\hat{\partial}S)$ is a subset $\mathcal{O} \subset E(\hat{\partial}S)$ that contains exactly the right ends for the given orientation. We can collect all of this information in the following diagram:

\begin{figure}[h!]
\begin{center}
\begin{tikzpicture}[scale=1.5]

\node (A) at (0,.7 ) {$\pi_0(\hat{\partial}S)$};
\node (B) at (1.35,.7) {$E(\hat{\partial}S)$};
\node (C) at (2.6,.7) {$E(S)$};
\node (D) at (1.35,0) {$\mathcal{O}$};
\node (E) at (3.9, .7) {$E_{\infty}(S)$};
\node (F) at (2.6, 0) {$E_{\partial}(S)$};

\path[->,font=\scriptsize,>=angle 90]
(B) edge node[above]{$e$} (A)
(B) edge node[above]{$v$} (C)
(D) edge node[right]{} (B)
(E) edge node[right]{} (C)
(F) edge node[right]{} (C);
\end{tikzpicture}
\end{center}
\caption{A surface diagram.}
\label{surfacediagram}
\end{figure}

The unlabeled arrows are the inclusion maps. We will refer to this as the \textit{surface diagram} for the surface $S$. See \cite{classification} for the construction of a surface from a given \textit{abstract surface diagram} which is a diagram of the above form consisting of topological spaces and maps satisfying various technical conditions. The abstract surface diagram provides a bundle of data whose homeomorphism types are in correspondence with the homeomorphism types of surfaces. Here we consider diagrams to be homeomorphic when there are homeomorphisms between each of the sets which commute with the arrows. We will not use abstract surface diagrams in this paper, so we leave it to the reader to review this definition if desired. One should also note that for the nonorientable case there is extra data to consider which is not represented in the above diagram.

       
       
       
       
%
    
\begin{theorem} \label{classification} (Classification of Surfaces, \cite{classification})
   Two surfaces are homeomorphic if and only if they have homeomorphic surface diagrams, the same genus, and the same number of compact boundary components.
\end{theorem}

Since the general case is vastly more complicated, we give a few illustrative examples some of which were discussed in the introduction of \cite{classification}.

\begin{example} \label{nonhomeoex}
See Figure \ref{nonhomeo}. The two surfaces shown have homeomorphic ends spaces $E(S) = E_{\infty}(S) = \omega \cdot 2 + 1$. Notice the doubles of these surfaces are homeomorphic. Here the double of a surface with boundary is constructed by taking two copies and gluing along the boundary by the identity. However, the surfaces themselves are not homeomorphic since they have nonhomeomorphic diagrams. To see this note that the upper surface has a noncompact boundary component such that both ends get sent by $v$ to accumulation points of $E(S)$, but the lower surface does not. It follows that there cannot be homeomorphisms between their $E(\hat{\partial} S)$ and $E(S)$ sets which commute with the $v$ maps.
\end{example}
\begin{figure}[ht]
\begin{center}
\begin{tikzpicture}
\node[anchor=south west,inner sep=0] at (0,0) {\includegraphics[scale = 0.23]{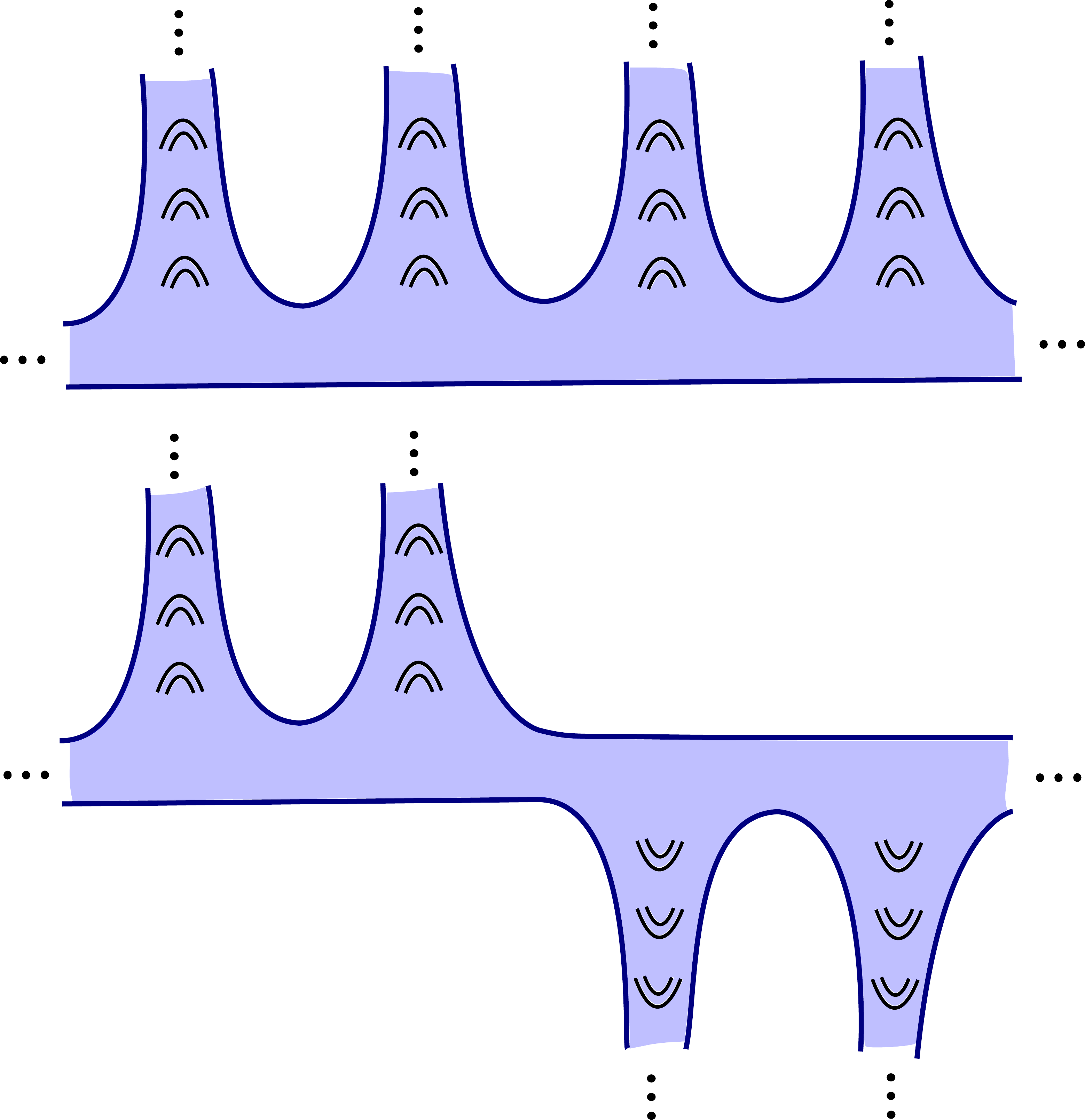}};

\end{tikzpicture}
\end{center}
\caption{Nonhomeomorphic surfaces with homeomorphic doubles. The boundary components are represented by the blue lines.}
\label{nonhomeo}
\end{figure}

\begin{example} \label{nonhomeo2}
Take an annulus and from each boundary component remove a point and a sequence accumulating to the point monotonically. There is a choice whether both sequences converge in the same direction or not, and this gives two nonhomeomorphic surfaces. These surfaces have homeomorphic end spaces $E(S) = \omega \cdot 2 + 1$, and even the top rows of their surface diagrams are homeomorphic. The full diagrams are not homeomorphic, however, because the orientations disagree. When the sequences go in the same direction then either $\mathcal{O}$ or  $E(\hat{\partial} S) \setminus \mathcal{O}$ contains (the preimages of) both of the accumulation points, but when the sequences go in opposite directions then $\mathcal{O}$ contains exactly one of the accumulation points. Similar reasoning gives another explanation why the surfaces in Example \ref{nonhomeoex} are nonhomeomorphic. 

This example highlights an interesting distinction from the compact boundary case. A connected sum of surfaces with compact boundary always gives a unique surface, up to homeomorphism, but for general surfaces there may be at most two homeomorphism types depending on the orientation of the attaching map. The above two surfaces are each the connected sum of the same disks with boundary points removed. For orientable surfaces, a connected sum determines a unique surface if and only if at least one of the surfaces has an orientation-reversing self-homeomorphism. 
\end{example}

Now we define a class of surfaces essential to this paper. By attaching a handle or tube to a surface we mean removing two open balls with disjoint closures and then identifying the resulting boundary components by an orientation reversing map of degree -1. 

\begin{definition} (Disk with Handles) \label{disk with handles}
A \textit{Disk with Handles} is a surface which can be constructed by taking a disk, removing a closed embedded subset $\mathcal{P}$ of the Cantor set from the boundary, and then attaching infinitely many handles accumulating to some subset of $\mathcal{P}$. The choice of infinitely many handles was chosen to simplify the statement of the theorems and to remove finite type surfaces.
\end{definition}

\begin{figure}[ht]
\begin{center}
\begin{tikzpicture}
\node[anchor=south west,inner sep=0] at (0,0) {\includegraphics[scale = 0.410]{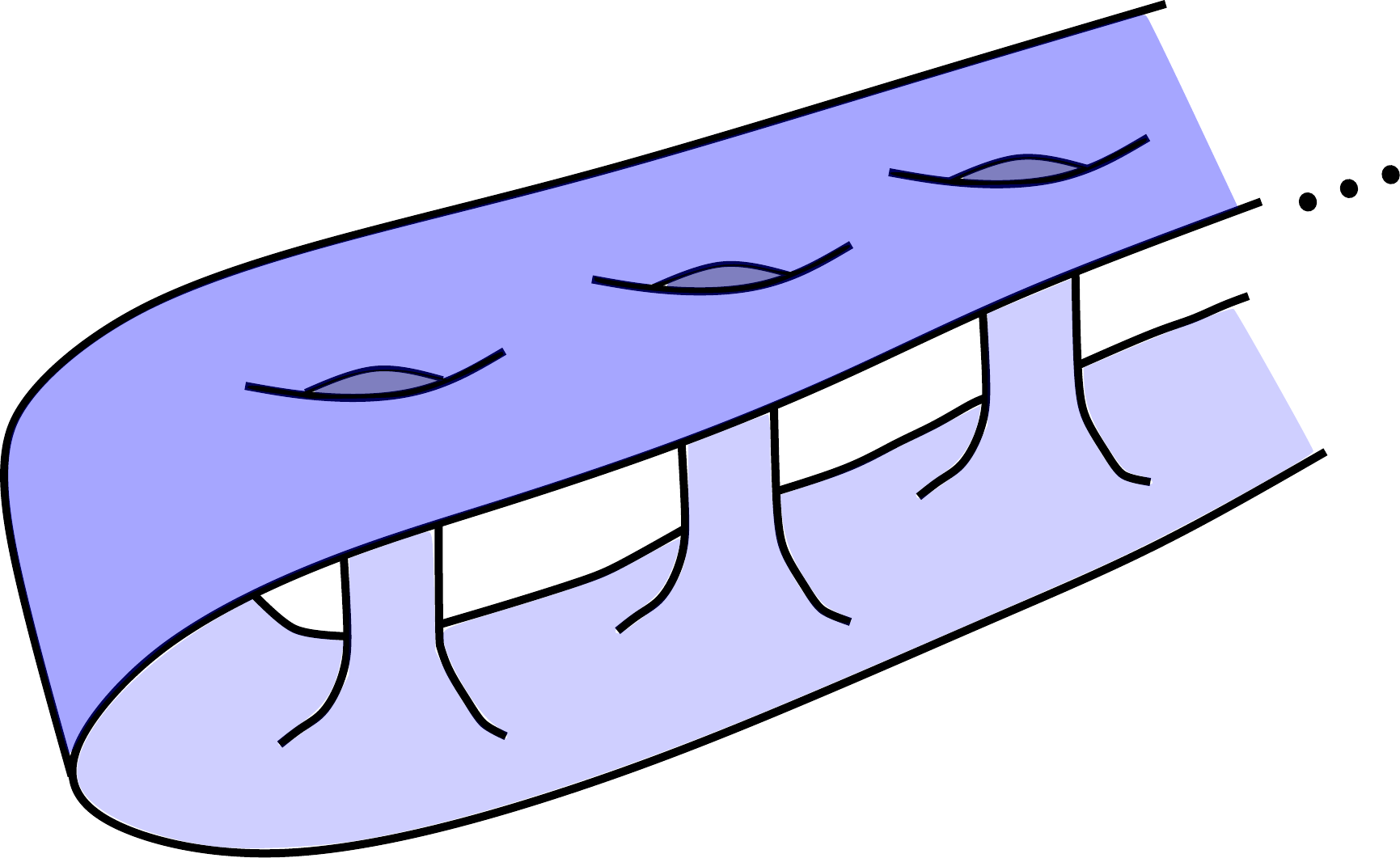}};

\end{tikzpicture}
\end{center}
\caption{A visualization of the 2-Sliced Loch Ness Monster.}
\label{doublysliced}
\end{figure}

\begin{remark} \label{attachingtubes}
Let $D$ be a disk with boundary points removed, and $S$ a Disk with Handles constructed from $D$. When we attach a sequence of handles to $D$, it is possible the two corresponding sequences of open balls accumulate to different points in $E(D)$. This joins these ends into a single end of $E(S)$. This is highlighted in Construction \ref{construction} and Figures \ref{doublysliced} and \ref{dwh1}. Due to this phenomenon, a general Disk with Handles is much more complicated than one might first expect. 

If we assume this type of handle attaching does not occur, then the possible Disks with Handles are classified by homeomorphism types of the pair $(E(S), E_\infty(S))$ with the additional structure of a cyclic ordering. Note that this gives another way to distinguish the surfaces in Example \ref{nonhomeoex}.  A more complicated type of ordering allowing repeats is required to classify general Disks with Handles. A major part of the Brown-Messer construction for a surface from a given diagram involves the delicate construction of such orderings \cite{classification}.
\end{remark}

Now we also want to consider a more specific class of surfaces.

\begin{definition} \label{slicedlochness} (Loch Ness Monsters) \label{slicedlochnessex} 
A \textit{Loch Ness Monster} refers to the unique surface with one end, infinite genus, and empty boundary. A \textit{Sliced Loch Ness Monster} is any of the surfaces with one end, infinite genus, no compact boundary components, but at least one noncompact boundary component. Equivalently, a Sliced Loch Ness Monster is a Disk with Handles with one end. By the Classification of Surfaces, a surface with these properties is determined by the possibly infinite number of noncompact boundary components. We sometimes refer to an $n$-Sliced Loch Ness Monster to emphasize the number of boundary components.
\end{definition}

In order to help visualize these surfaces, we give the following construction. 

\begin{construction} \label{construction}
Take a strip $\R \times [-1,1]$, and remove small disjoint open balls centered around the points $(n, 0)$ for $n \in \Z \setminus \{0\}$. Now identify pairs of boundary components centered at $(\pm n, 0)$ via horizontal reflection. Equivalently, we may view this process as attaching tubes to the strip. The resulting manifold is the 2-Sliced Loch Ness Monster. See Figure \ref{doublysliced} for a visualization. Similarly, we can construct the $n$-Sliced Loch Ness Monster for any finite $n$ by taking a disk with $n$ points removed from the boundary and attaching tubes to join all of the ends. To get the $\infty$-Sliced Loch Ness Monster, we can take a disk with any infinite embedded closed subset of the Cantor set removed from the boundary and attach tubes to join all of the ends as before. By the Classification of Surfaces, no matter what infinite set of points we remove in this construction we always get the same surface. This is somewhat nonintuitive, but it is better understood once we realize that any interesting topology in the original ends space is collapsed when we attach tubes to get a surface with a single end.

\end{construction}

\begin{figure}[htb]
\begin{center}
\begin{tikzpicture}
\node[anchor=south west,inner sep=0] at (0,0) {\includegraphics[scale = 0.175]{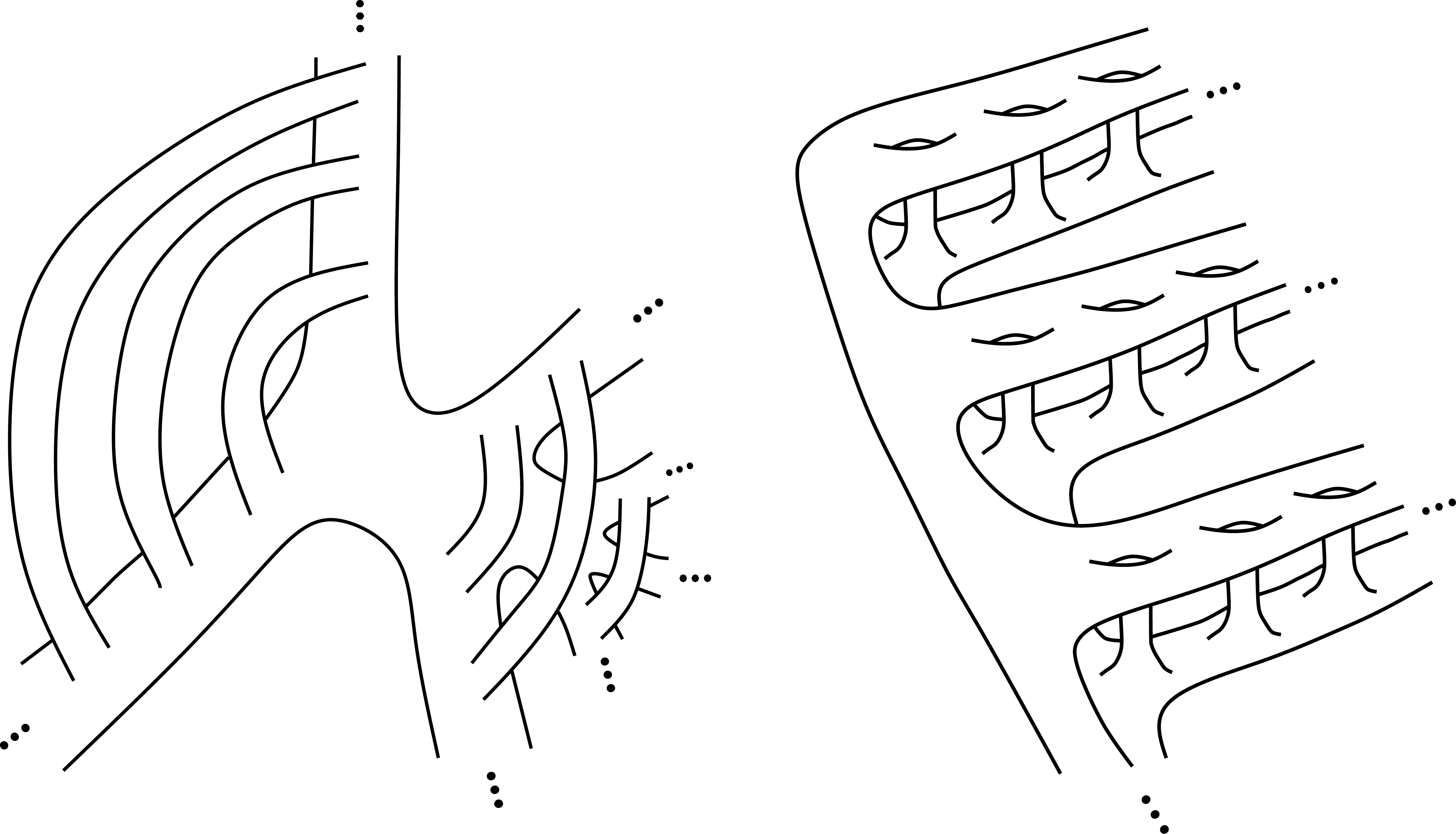}};

\end{tikzpicture}
\end{center}
\caption{Two visualizations of the same Disk with Handles. Note that this surface can be cut along arcs into infinitely many 2-Sliced Loch Ness Monsters.}
\label{dwh1}
\end{figure}

The choice to define Sliced Loch Ness Monsters independent of the number of boundary components simplifies the statement of the main theorems. Also for the first statement of Theorem C it will be necessary to include Sliced Loch Ness Monsters which have any number of noncompact boundary components going out the single end. See Figure \ref{dwh1} for an example of a Disk with Handles which suggests that we must include 2-Sliced Loch Monsters in the list of building blocks. 

According to Theorem C, an infinite type surface with every end accumulated by genus can be cut along curves into Loch Ness Monsters and Disks with Handles (without planar ends) each possibly with compact boundary components added. Therefore, this class of surfaces corresponds to the set of all surfaces which result from a possibly infinite procedure of connected sum operations with these building blocks. In Remark \ref{extendc} we discuss a possible extension of Theorem C to general surfaces possibly with finite genus and planar ends. In this case, we must allow more building blocks, in particular disks with boundary points removed and possibly compact boundary components added or finitely many handles attached. Many basic examples one should consider involve inductive procedures of connected sum operations with these building blocks. Note that Theorem C or any version can only tell us that some procedure exists for connecting together building blocks to create a general surface, but this procedure may not necessarily be describable in an inductive manner.



\subsection{Big Mapping Class Groups}

Let $\Homeo^+_\partial(S)$ be the group of orientation preserving homeomorphisms of a surface $S$ which fix the boundary pointwise. The \textit{mapping class group} $\MCG(S)$ is defined to be \begin{align*}
    \MCG(S) = \Homeo_{\partial}^{+}(S)/\sim
\end{align*} where two homeomorphisms are equivalent if they are isotopic relative to the boundary of $S$. We will often conflate a mapping class group element with a representative homeomorphism. We equip $\Homeo^+_\partial(S)$ with the compact-open topology, which induces the quotient topology on $\MCG(S)$. We equip subgroups of $\MCG(S)$ with the subspace topology. The mapping class group of a subsurface will correspond to the subgroup of elements which have a representative supported in the subsurface. The \textit{pure mapping class group} $\PMCG(S)$ is the subgroup of $\MCG(S)$ consisting of elements which fix the ends of $S$. 

We say $f \in \MCG(S)$ is \textit{compactly supported} if $f$ has a representative that is the identity outside of a compact subsurface of $S$. The subgroup consisting of compactly supported mapping classes is denoted by $\PMCG_c(S)$. Note any compactly supported mapping class is in the subgroup $\PMCG(S)$.

\begin{definition} \label{degenerate} (Degenerate ends) Notice removing an embedded closed subset of the Cantor set from the boundary of a surface does not change the underlying mapping class group. We refer to the resulting ends as \textit{degenerate}. More generally, this will refer to ends with a representative sequence $\{U_i\}$ such that some $U_i$ is homeomorphic to a disk with boundary points removed. It may be convenient in some cases to only work with homeomorphism types of surfaces up to filling in the degenerate ends. We will allow these ends except when stated otherwise. Note that given the definition of a finite type surface from the introduction, there can be finite type surfaces with infinitely many degenerate ends.
\end{definition}

Now we review the definition of a handle shift from \cite{APV2017} which will be used throughout Section 6. Let $\Sigma$ be the surface obtained by gluing handles onto $\R \times [-1,1]$ periodically with respect to the map $(x,y) \mapsto (x+1,y)$. We refer to this surface as a $\textit{Strip with Genus}$. For some $\epsilon>0$, let $\sigma:\R \times [-1,1] \rightarrow \R \times [-1,1]$ be the map determined by setting
\begin{align*}
    \sigma(x,y) &= (x+1,y) \;\; \text{ for } (x,y) \in \R\times[-1+\epsilon,1-\epsilon], \\
    \sigma(x,y) &= (x,y) \;\; \text{ for } (x,y) \in \R\times\{-1,1\},
\end{align*}
and interpolating continuously on $\R\times[-1,-1+\epsilon] \cup \R\times[1-\epsilon,1]$. By extending this map to the attached handles, we get a homeomorphism on $\Sigma$ which we conflate with $\sigma$. A homeomorphism $h : S \rightarrow S$ is a \textit{handle shift} if there exists a proper embedding $\iota: \Sigma \rightarrow S$ such that 
\begin{align*}
    h(x) = 
    \begin{cases} 
        (\iota \circ \sigma \circ \iota^{-1})(x) & \text{ if } x \in \iota(\Sigma) \\
        x & \text{ otherwise } 
    \end{cases}.
\end{align*}

The embedding $\iota$ is required to be proper, so it induces a map $\hat{\iota}: E(\Sigma) \to E_\infty(S)$. A handle shift $h$ then has an attracting and repelling end denoted by $h^+$ and $h^-$, respectively. In general, the attracting and repelling ends can be the same, though the handle shifts used in Section 6 will have different attracting and repelling ends.

\subsection{Curves and Arcs}

A simple closed curve in a surface $S$ is the image of a topological embedding $\mathbb{S}^1 \hookrightarrow S$. A simple closed curve is trivial if it is isotopic to a point; it is peripheral if it is
either isotopic to a boundary component or bounds a once-punctured disk. We will often refer to a simple closed curve as just a curve.

An arc in $S$ is a topological embedding $\alpha: I \hookrightarrow S$, where $I$ is the closed unit interval, and $\alpha(\partial I) \subset \partial S$. We consider all isotopies between arcs to be relative to $\partial I$; i.e., the isotopies are not allowed to move the endpoints. An arc is trivial if it is isotopic to an arc whose image is completely contained in $\partial S$; it is peripheral if it bounds a disk with a single point removed from the boundary. This last definition is the only nonstandard one, and we include it since it aligns with the definition of a peripheral curve.  It may be useful in some cases to extend the definition of trivial/peripheral to include arcs or curves which are trivial/peripheral in the surface after degenerate ends are filled in. 

A curve or arc is essential when it is not trivial nor peripheral; it is separating if its complement is disconnected and nonseparating otherwise. We will often conflate a curve or arc with its isotopy class. All curves and arcs will be assumed to be essential unless stated otherwise. We say curves or arcs intersect if they cannot be isotoped to be disjoint, and we say they are in minimal position when they are isotoped to have the smallest number of intersections. We say a subsurface is essential if the inclusion of the subsurface induces an injective map of fundamental groups.

By cutting along a collection of curves or arcs, we mean removing disjoint open regular neighborhoods of each of the curves or arcs. Throughout this paper, we will conflate the complement of a curve or arc with this cut surface. We will also occasionally conflate the complement of a subsurface with its closure.

\section{Decomposing an Infinite Type Surface} \label{decomp}

\subsection{Outline}

In this section, we prove Theorem C along with several other decomposition results. This is crucial for the proof of the main theorems since a general surface can be extremely complicated. We also want an approachable method for visualizing a surface from the surface diagram data. Our work builds off the Brown-Messer classification \cite{classification} with some inspiration from \cite{Prish07}. The classification theorem from the latter paper is incorrect as stated. The theorem can not distinguish the pairs of surfaces from Example \ref{nonhomeoex} and Example \ref{nonhomeo2}. On the other hand, the argument from this paper does provide a more intuitive approach. We precisely define some of the ideas from this paper. 

The main idea is to study what happens when we remove the boundary of a surface $S$. Deleting a compact boundary component leaves a puncture which corresponds to an isolated end of $S^{\mathrm{o}}$, the interior of $S$. Deleting the noncompact boundary components is more complicated as there could be several ends corresponding to these boundary components which get sent to a single end. 

We show that we can think of the noncompact boundary components and their ends as being grouped together into chains, and that removing the boundary components from a chain sends all of the corresponding ends to a single end of $S^\lilo$. An important tool will be Lemma 4.12 which allows us to cut a surface along curves so each resulting component has at most one boundary chain. After we discuss the types of surfaces which have a single boundary chain (see Lemma 4.8 and the remarks at the end of its proof), we can apply Lemma 4.12 to prove the furthermore statement of Theorem C by representing the components with boundary chains as Disks with Handles possibly with compact boundary components added. The boundary chains will then correspond to the boundaries of the Disks with Handles.

\subsection{Boundary Ends and Chains}
First we want to precisely define the map on ends spaces induced by deleting the boundary. Consider the inclusion of $S$ in $S^\prime = S \cup_{\partial S} (\partial S \times [0, \infty))$. Notice that $S^\prime$ is homeomorphic to $S^\lilo$. Consider a compact exhaustion $\{S^\prime_i\}$ of $S^\prime$. Let $S_i = S^\prime_i \cap S$ so $\{S_i\}$ is a compact exhaustion of $S$. Choose an end in $E(S)$, and let $\{U_i\}$ be an exiting sequence representative for this end consisting of complementary domains of the $S_i$. By replacing components of the $S \setminus S_i$ with components of the $S^\prime \setminus S^\prime_i$, we can get an exiting sequence in $S^\prime$. It follows that we have a well-defined canonical map $$\pi: E(S) \to E(S^\lilo)$$

\begin{proposition} \label{picontinuous}
The map $\pi$ is continuous.
\end{proposition}

\begin{proof}
Let $U^\star \subseteq E(S^{\mathrm{o}})$ be a basis element defined by some complementary domain $U$ in $S^{\mathrm{o}}$. This gives a complementary domain $U_S$ in $S$ after adding in the boundary, and so it defines a basis element $U_S^\star \subseteq E(S)$. We are done after noting that $\pi^{-1}(U^\star) = U_S^\star$.
\end{proof}

Now recall the definition of a surface diagram from Section 3 (see Figure \ref{surfacediagram}). Let $V$ be the image of $v$, the map which sends ends of noncompact boundary components to ends of $S$. Now we use the map $\pi$ to define a boundary chain. Intuitively, this can be thought of as a set of boundary ends and boundary components which can be realized in the surface as a circle with points removed. 

\begin{definition} (Boundary Chains) 
A \textit{boundary chain} of a surface $S$ is a subset of $E(S)$ of the form $\pi^{-1}(p)$ where $p \in \pi(V)$. The collection of all such sets is denoted by $C(S)$ and is referred to as the \textit{set of boundary chains} for $S$. Occasionally, we will conflate definitions and use boundary chain to refer to the union of noncompact boundary components with ends in the chain. 
\end{definition}

\begin{figure}[!htbp]
\begin{center}
\begin{tikzpicture}
\node[anchor=south west,inner sep=0] at (0,0) {\includegraphics[scale = 0.32]{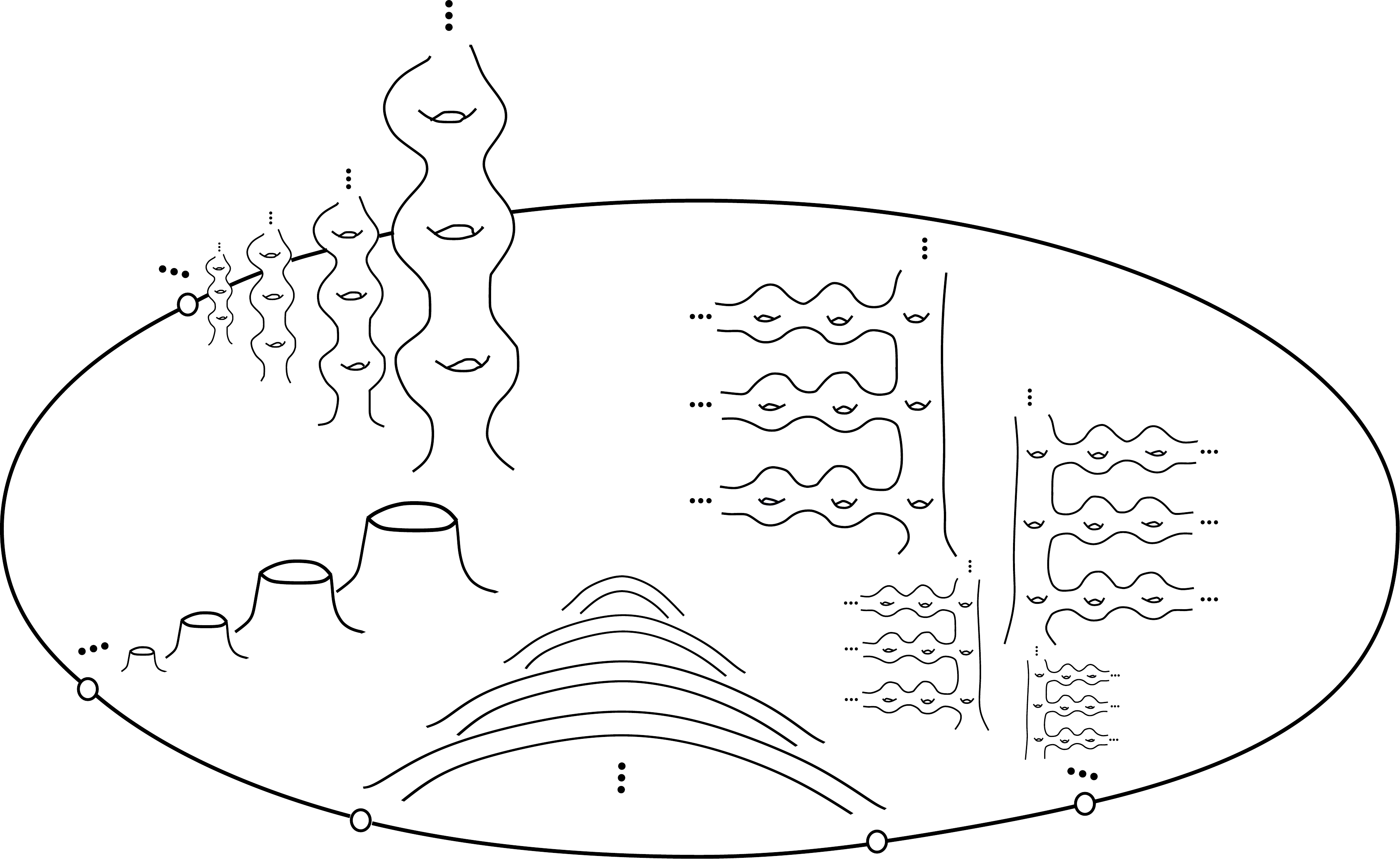}};

\end{tikzpicture}
\end{center}

\label{singlechain}

\caption{A surface with a single boundary chain.}

\end{figure}

Now we define the set of boundary ends for a surface.

\begin{definition} \label{boundaryends} (Boundary Ends) 
Let $B(S)$ be the union of the boundary chains. This will be referred to as the \textit{set of boundary ends}, and any element of $B(S)$ is a \textit{boundary end}.
\end{definition}

An end in $E(S)$ is said to be an \textit{interior end} if it is not in $B(S)$. If a boundary end in $S$ is isolated from the other ends, then we refer to it as a \textit{boundary puncture}. Note that $B(S)$ contains $V$, but it is possible that $B(S)$ contains additional ends. The definitions above were specifically chosen to include additional ends such as the ones from the following example.

\begin{example} \label{diskminuscantor}
Consider a disk with a Cantor set removed from the boundary. We want every end of this surface to be considered a boundary end, but there are some ends which are not in the image of $v$. These correspond to points in the Cantor set which are not the endpoint of any interval that is removed during the usual middle thirds construction.
\end{example}

We can use $\pi$ to define an equivalence relation on $B(S)$ for which $C(S)$ is the resulting quotient. After equipping $B(S)$ with the subspace topology, $C(S)$ inherits the quotient topology. Note $\pi$ is injective on $E(S) \setminus B(S)$. The set of boundary chains exactly records the noninjectivity of $\pi$ on $B(S)$.

\begin{remark} \label{countable}
Since there are countably many boundary components in a surface, $\pi(B(S))$ is countable. 
\end{remark}

\begin{remark} \label{notclosed}
The subset $B(S) \subseteq E(S)$ is not necessarily closed. For example, take a once-punctured sphere, remove infinitely many open balls with disjoint closures accumulating to the puncture, and then remove a single point from each of the resulting boundary components. It is not necessarily open either as in the case of a disk with a point removed from the boundary with interior punctures added accumulating to the boundary end. By Proposition \ref{picontinuous}, each boundary chain is a closed subset. Then by Remark \ref{countable}, $B(S)$ is the countable union of closed subsets. 
\end{remark}

\begin{example} 
Consider any Disk with Handles $S$. The interior of $S$ is the Loch Ness Monster since it corresponds to a once-punctured sphere with handles attached accumulating to the puncture. Every boundary end of $S$ gets sent by $\pi$ to the single end of the Loch Ness Monster, so any Disk with Handles has a single boundary chain.
\end{example}

This last statement has a partial converse which provides a more intuitive way to think about a boundary chain. 

\begin{lemma}\label{chaintodisk}
Every surface $S$ with infinite genus, one boundary chain, and only boundary ends is a Disk with Handles possibly with compact boundary components added.
\end{lemma}

Before we prove Lemma \ref{chaintodisk} we first need a few facts.

\begin{proposition} \label{oneendequiv}
Let $S$ be a noncompact surface without boundary. Then the following are equivalent:

\begin{enumerate} [(i)]
    \item There exists a compact exhaustion $\{S_i\}$ of $S$ such that each $\partial S_i$ has a single component.
    \item $S$ has exactly one end.
\end{enumerate} 
\end{proposition}

\begin{proof}
Since the complementary regions of a compact exhaustion can be used to build exiting sequences, and the ends space is independent of this choice of a compact exhaustion, the first condition immediately implies the second. Assuming the second condition, $S$ is either a finite type surface with one puncture, or the Loch Ness Monster. In either case, we can directly construct the desired exhaustion.  
\end{proof}

\begin{proposition} \label{singlechainequiv}
 Let $S$ be a noncompact surface with no compact boundary components and no interior ends. Then the following are equivalent:
 
 \begin{enumerate} [(i)]
     \item There exists a compact exhaustion $\{S_i\}$ of $S$ such that each $\partial S_i$ has a single component.
     \item $S$ has exactly one boundary chain. 
 \end{enumerate}
\end{proposition}

\begin{proof}
Suppose the first condition holds. To get the second condition it suffices to show that the interior of $S$ has a single end. Remove open regular neighborhoods of the boundary from each $S_i$, shrinking the neighborhoods as we increase $i$ so we get a compact exhaustion for the interior. Each subsurface in this exhaustion has one boundary component, so we are done by Proposition \ref{oneendequiv}.

Now suppose the second condition holds, so that the interior of $S$ has a single end by the definition of a boundary chain. Let $\{K_i\}$ be a compact exhaustion of the interior of $S$ given by Proposition 4.9 such that each $\partial K_i$ has a single boundary component. Let $N$ be an open regular neighborhood of the boundary chain. Note if we set $S_i = K_i \cap (S \setminus N)$, then we get a compact exhaustion $\{S_i\}$ for $S \setminus N$.

We want to modify the $K_i$ so the resulting $S_i$ each have a single boundary component. First remove subsurfaces from $\{K_i\}$ so $K_1 \cap N \ne \emptyset$. Now isotope $\partial K_1$ so it is transverse to $\partial \overline{N}$, and each component of $K_1 \cap \overline{N}$ is a bigon. Now we proceed inductively. Remove some subsurfaces from the exhaustion so $K_i$ contains the previously modified $K_{i-1}$, and isotope $\partial K_i$ in $S \setminus K_{i-1}$ so its position with $\overline{N}$ is as above. We can ensure the bigons exhaust the interior of $N$, so the modified sequence $\{K_i\}$ is an exhaustion of the interior of $S$. Now since $S_i$ is the result of removing disjoint bigons from $K_i$, we conclude that each $S_i$ has one boundary component. We are now done since $S \setminus N$ is homeomorphic to $S$.
\end{proof}

Now we are ready to prove Lemma  \ref{chaintodisk}. 

\begin{proof} [Proof of Lemma \ref{chaintodisk}]
 Throughout this proof, we modify $S$ also calling the new surface at each step $S$. First cap any compact boundary components of $S$ with disks. Since $S$ has no interior ends, one boundary chain, and now no compact boundary components, Proposition \ref{singlechainequiv} gives us a compact exhaustion $\{S_i\}$ of $S$ such that each $\partial S_i$ has one component. Now we want to find an infinite sequence of nonseparating curves such that cutting $S$ along the curves gives a surface with no genus and one boundary chain. See Figure 7 for an example. We must be careful since cutting the surface from this figure along a sequence of horizontal curves about each tube similar to the two leftmost curves gives two surfaces each with one boundary chain. If we exclude the curve about the middle tube, then cutting gives a surface with two boundary chains.
 
 To get the desired collection of curves, note we can find a finite collection of curves in each $S_i$ which cut it into a surface with no genus and one boundary component (after capping the compact boundary components resulting from cutting), and we can ensure each collection extends to subsequent collections. The desired collection of curves is then the increasing union of these collections. Cut $S$ along these curves, and cap the resulting boundary components with disks. Now $S$ has no genus, and by applying Proposition \ref{singlechainequiv} to a modified compact exhaustion we see $S$ has one boundary chain.
 
As in the first paragraph of the proof of Proposition 4.10, by removing open regular neighborhoods from the boundary of the $S_i$, we can get a compact exhaustion of the interior of $S$ satisfying the first condition of Proposition \ref{oneendequiv}. Then by Proposition \ref{oneendequiv} and Classification of Surfaces, the interior of $S$ is homeomorphic to a once-punctured sphere. Therefore, if we fill in the boundary ends of $S$, we get a compact surface which must be homeomorphic to a disk. We are then done after reversing the above steps since this will correspond to deleting points from the boundary and then attaching handles as in the definition of a Disk with Handles. Finally, if there were initially any compact boundary components then reversing the capping corresponds to adding back in these components.

We should mention that a version of this lemma holds if we allow surfaces with finite genus. In this case, our surface will be homeomorphic to a disk with boundary points removed with finitely many (possibly zero) handles attached and possibly compact boundary components added. We could also allow interior ends, and then we would need to allow a final step where we delete interior points from the modified disk and then possibly attach handles or compact boundary components accumulating to any of the ends. The overall takeaway of this lemma is that a surface with a single boundary chain is homeomorphic to a modified disk.
\end{proof}

\begin{figure}[!htbp]
\label{ladder}
\begin{center}
\begin{tikzpicture}
\node[anchor=south west,inner sep=0] at (0,0) {\includegraphics[scale = 0.25]{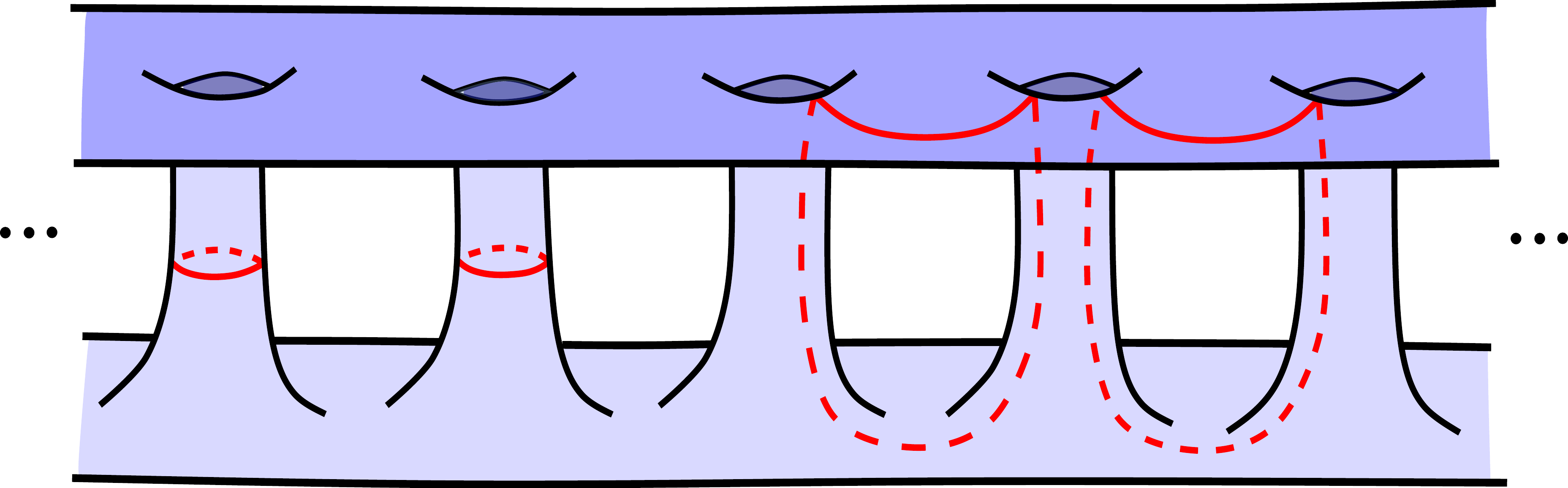}};

\end{tikzpicture}
\end{center}

\caption{A surface satisfying the conditions of Lemma \ref{chaintodisk} with a collection of curves that cuts it into a surface with zero genus and one boundary chain.}
\end{figure}

\subsection{Decomposition Results}

\begin{lemma} \label{noboundarydecomp}
Every infinite type surface $S$ without boundary and without planar ends can be cut along a collection of disjoint essential simple closed curves into Loch Ness Monsters with $k \in \N \cup \{\infty\}$ compact boundary components added.
\end{lemma}

\begin{proof} 
 This was first shown in \cite{APV2017} as a tool to prove Theorem \ref{structure}. See Section \ref{APV} for this argument. We provide a different proof which gives us more control over the ends of the components in the cut surface. Recall that the ends space $E(S)$ is homeomorphic to a closed subset of the Cantor set. Let $T$ be some locally finite tree with $E(T)$ homeomorphic to $E(S)$.\footnote{The ends space of a tree is defined analogously to the ends space of a surface. For locally finite trees, the ends space is always homeomorphic to a closed subset of the Cantor set.} We can think of $S$ as a thickened version of $T$ with genus added accumulating to every end. For simplicity, we will assume $T$ has no vertices of valence one.
 
We may write $T$ as a union of rays $\{R_i\}$ where for each distinct $R_i$ and $R_j$, $R_i \cap R_j$ is empty or a single vertex. To see this enumerate a countable dense subset $\{x_i\}$ of $E(T)$, and fix some basepoint vertex $v$. Begin by letting $R_1$ be the ray from $v$ to $x_1$, then let $R_2$ be the ray from $v$ to $x_2$ with the interior of the overlap with $R_1$ deleted. Continue in this manner to build the desired collection $\{R_i\}$. Since $T$ has no vertices of valence one, this will exhaust the entire tree. 

Associate each $R_i$ with a Loch Ness Monster $L_i$. Let $n_i \in \N \cup \{\infty\}$ be the number of vertices in $R_i \cap \bigcup_{j \ne i} R_j$. For each $i$, remove $n_i$ open balls with disjoint closures from $L_i$ with the balls accumulating to the single end when $n_i$ is infinite. Associate each boundary component of $L_i$ with a vertex in $R_i \cap \bigcup_{j \ne i} R_j$, and attach the boundary components of distinct $L_i$, $L_j$ when these components correspond to the same vertex of $T$. Let $S^\prime$ be the resulting surface, and let $\{\alpha_i\}$ be the collection of curves in $S^\prime$ corresponding to the attached boundary components. 

Now $E(T)$ and $E(S^\prime)$ are homeomorphic. This requires showing a correspondence between a compact exhaustion of $T$ and an exhaustion of $S^\prime$. One approach is to subdivide $T$, then write it as a union of stars of the vertices from the original tree. Then associate each star with an $n$-holed torus where $n$ is the number of edges in the star. The stars and the tori can then be attached to build compact exhaustions for $T$ and $S^\prime$, respectively. By the Classification of Surfaces, $S^\prime$ is homeomorphic to $S$. Cutting $S^\prime$ along the $\alpha_i$ gives components which are Loch Ness Monsters with compact boundary components added, so we are done.
 \end{proof}

The argument from this lemma will be referenced often in the following proofs. By decomposing a tree $T$ into rays we mean writing $T$ as a union of rays which are either disjoint or intersect one another at a single vertex.

\begin{lemma} \label{cutchains}
Every surface $S$ can be cut along a collection of disjoint simple closed curves into components with at most one boundary chain. Additionally, we may assume the components with boundary chains have only boundary ends.
\end{lemma}

\begin{proof}
We can assume $S$ has noncompact boundary components since otherwise the lemma holds trivially. Recall that by the definition of a boundary chain, two boundary ends $p,q \in B(S)$ are in the same boundary chain if and only if $\pi(q) = \pi(p)$. Suppose we can cut $S^\lilo$ along curves into one-ended components so that $\pi(B(S))$ is contained in the dense subset of $E(S^\lilo)$ corresponding to these components. Now when we cut $S$ along these same curves, each component of the cut surface has at most one boundary chain. Since $\pi$ maps interior ends outside of $\pi(B(S))$, we get the last statement of the lemma. Therefore, it suffices to decompose $S^\lilo$ in this manner.

Following the proof of Lemma \ref{noboundarydecomp}, represent $S^\lilo$ by a tree $T$ with no valence one vertices. Fix a base vertex $v$, and let $T^\prime$ be the union of rays from $v$ to an end in $\pi(B(S))$. By Remark \ref{countable}, $\pi(B(S))$ is countable, so enumerate the elements of $\pi(B(S))$ as a sequence $\{x_i\}$. As in the proof of Lemma \ref{noboundarydecomp}, we can use an inductive process to decompose $T^\prime$ into a collection of rays $\{R_i\}$ where each element of $\{x_i\}$ corresponds to the end of one of the rays. Then we can decompose the remainder of $T$ into rays. Now we follow the proof of Lemma \ref{noboundarydecomp} to decompose $S^\lilo$ as desired. Note for this last step we need to allow one-ended pieces with finite genus into our decomposition since we are not assuming $S$ has only ends accumulated by genus (see Remark \ref{extendc}). We may also need to allow nonessential curves.

\end{proof}

\begin{lemma} \label{diskdecomp}
Every Disk with Handles $S$ without planar ends can be cut along a collection of disjoint essential arcs into Sliced Loch Ness Monsters. 
\end{lemma}

\begin{proof}

Let $D$ be a disk with points removed from the boundary used to construct $S$. Note we may realize $D$ as a closed neighborhood of a tree $T$ properly embedded in $\C$.\footnote{One approach is to take a triangulation of $D$, and then build $T$ from a spanning tree of the dual 1-skeleton.} As before we will assume this tree has no valence one vertices. Recall from Remark \ref{attachingtubes} that the handles may be attached in a way that joins ends of $D$ together. By similar reasoning to Proposition \ref{picontinuous} and the preceding remarks, the process of attaching handles determines a well-defined continuous quotient map $$q: E(D) \rightarrow E(S)$$ By Classification of Surfaces, two Disks with Handles without planar ends are homeomorphic when there is a homeomorphism between the base disks which respects the quotient maps induced by attaching the handles.

We argue by analogy to the proof of Lemma \ref{noboundarydecomp}. First suppose that $q$ is injective. Decompose $T$ into rays $\{R_i\}$, and then associate each ray with a 1-Sliced Loch Ness Monster. Attach these surfaces along intervals on their boundaries according to the incidences of the $R_i$ in $T$. This attaching procedure is analogous to the procedure from Lemma 4.11 with boundary connected sum operations in place of the connected sum operations. This gives a Disk with Handles with a base disk homeomorphic to $D$ and an injective quotient map, so it is homeomorphic to $S$. It then follows that we can cut $S$ into 1-Sliced Loch Ness Monsters. See the left side of Figure \ref{cutexample} for an example of a Disk with Handles constructed from a thickened binary tree being cut into 1-Sliced Loch Ness Monsters. However, if $q$ is not injective, we need to be a little more careful. See for example the right side of Figure \ref{cutexample}. If we choose to cut this surface along arcs similar to the ones used for the left surface, then we will have components in the cut surface which are not Sliced Loch Ness Monsters. Let $$U = \{p \in E(S) \mid  |q^{-1}(p)|\geq 2\}$$ 
Enumerate a countable dense subset $\{x_i\}$ of $U$. Now when we decompose $T$ into rays, first choose rays that exhaust a dense subset of each $q^{-1}(x_i)$. Here we are conflating the ends space of $D$ with the ends space of $T$. Then decompose the remainder of $T$ to exhaust a dense subset of the entire ends space. Let $\{R_i\}$ be the resulting rays. Similar to before, associate each $R_i$ with a disk with one boundary puncture $D_i$, and attach the $D_i$ along intervals on their boundaries according to the incidences of the $R_i$ to get a base disk homeomorphic to $D$. 

\begin{figure}[ht]
\begin{center}
\begin{tikzpicture}
\node[anchor=south west,inner sep=0] at (0,0) {\includegraphics[scale = 0.70]{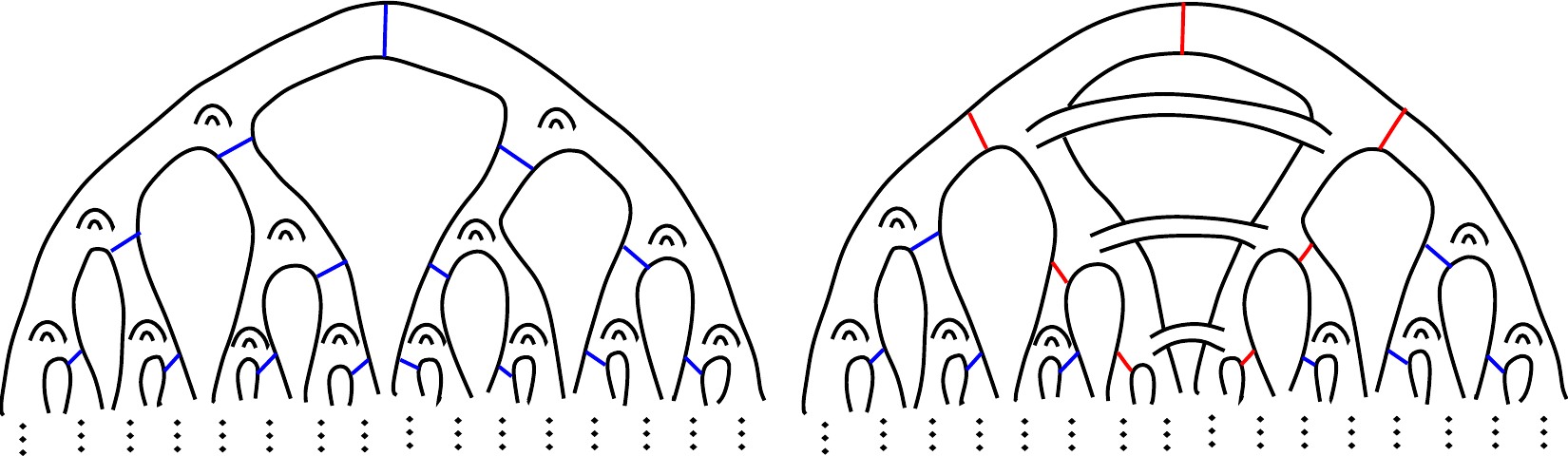}};

\end{tikzpicture}
\end{center}
\caption{Left - A Disk with Handles gets cut along blue arcs into 1-Sliced Loch Ness Monsters. Right - A more complicated Disk with Handles gets cut along the red and blue arcs into a 2-Sliced Loch Ness Monster and 1-Sliced Loch Ness Monsters. The surface bounded by the red arcs corresponds to two rays chosen to exhaust the subset $q^{-1}(x_1)$ where $x_1$ is the single element of $U$.}
\label{cutexample}
\end{figure}

Choose some $x_i$, and consider the subset of rays with an end corresponding to an element of $q^{-1}(x_i)$. Attach infinitely many handles to the union of the respective $D_j$ in order to join the boundary ends of the $D_j$ into a single end. Similar to Construction \ref{construction}, this gives $n$-Sliced Loch Ness Monsters where $n \geq 2$. Repeat this process for every $x_i$. Now for the remaining rays attach handles to the corresponding disks to get 1-Sliced Loch Ness Monsters. Attaching handles in this manner to the base disk gives an equivalence relation on $E(D)$ which agrees with the equivalence relation given by $q$ on a dense subset. Therefore by continuity and the above remarks, this construction gives a surface homeomorphic to $S$. Now we are done since cutting this surface along the $\alpha_i$ gives Sliced Loch Ness Monsters.
\end{proof} 

Now we combine everything thus far.

\begin{proof} [Proof of Theorem C]
 The first statement of this theorem is Lemma \ref{diskdecomp}. Let $S$ be an infinite type surface with infinite genus and no planar ends. Apply Lemma \ref{cutchains} to cut $S$ into components with at most one boundary chain where the components with a boundary chain have only boundary ends. We can assume each component has infinite genus since $S$ has only ends accumulated by genus. Then by Lemma \ref{chaintodisk}, the components with a boundary chain are Disk with Handles possibly with compact boundary components added. The other components are Loch Ness Monsters possibly with compact boundary components added.  
\end{proof}

\begin{remark} \label{extendc}
If we allow planar ends then similar decomposition results hold where we have to allow other one-ended building blocks. For example, when decomposing a Disk with Handles with planar ends similar to Lemma \ref{diskdecomp}, we need to include disks with one boundary puncture. We could also allow finite genus. For example, when decomposing a surface without boundary similar to Lemma \ref{noboundarydecomp}, we have to allow one-ended surfaces with finite genus and possibly with infinitely many compact boundary components added. In these cases we may need to allow cutting along peripheral curves and arcs. 

One possible extension of Theorem C to general surfaces with noncompact boundary involves using Lemma \ref{cutchains} and the extension of Lemma \ref{chaintodisk} mentioned in the final remarks of its proof. In this case, we must add the modified disks discussed in these remarks to our building blocks.
\end{remark}

\section{Main Results} \label{mainresults}

\subsection{Background}
Domat has shown for surfaces with compact boundary components and at least two ends accumulated by genus that $\overline{\PMCG_c(S)}$ is not perfect \cite{domat2020big}. In the appendix of this paper, the author and Domat use the Birman Exact Sequence to extend this to the case with one end accumulated by genus. On the other hand, Calegari has shown that the mapping class group of the sphere minus a Cantor set is uniformly perfect \cite{calegari}. Now we want to show many surfaces with noncompact boundary components have uniformly perfect mapping class groups.

First we need to extend a result of Patel-Vlamis to the general case, since we will use this implicitly throughout the proof of Theorem A. 

\begin{theorem} \cite{PV2018} \label{patelvlamis}
For any infinite type surface $S$ with only compact boundary components and at most one end accumulated by genus, $\PMCGcc{S} = \PMCG(S)$. 
\end{theorem}

This result was originally stated for compact boundary, but the proof in \cite{PV2018} also works when there are infinitely many compact boundary components. The argument uses pants decompositions which we can construct without adaptation when there are only compact boundary components. Pants decompositions seem more tedious to use in the general case, so we instead give a slightly modified proof using a more general exhaustion. To simplify our arguments we will assume surfaces do not have degenerate ends (Definition \ref{degenerate}) for the entirety of Section 5. Note this will not affect the proof of Theorem A since filling in degenerate ends does not change the mapping class group.

\begin{theorem} \label{patelvlamisextended}
For any infinite type surface $S$ with at most one end accumulated by genus, $\PMCGcc{S} = \PMCG(S)$. 
\end{theorem}

\begin{proof} 

Let $f \in \PMCG(S)$ be an arbitrary element. We want to find a sequence $\{f_i\}$ of elements of $\PMCG_c(S)$ such that $f_i \rightarrow f$ in the compact-open topology. Let $\{S_i\}$ be an exhaustion of $S$ by essential finite type surfaces. It will suffice to show that there is always some compactly supported $f_i$ which agrees with $f$ on $S_i$. We can assume that the complementary domains of each $S_i$ are infinite type. 

Note the orbit of any curve in $S$ under $\PMCG(S)$ is determined, up to isotopy, by the partition it determines on $E(S)$, the partition it determines on the compact boundary components of $S$, and the topological type of the complementary domains. The orbit of an arc, up to isotopy, is determined by the same properties and the endpoints of the arc. This is also true for curves and arcs in any surface.

Fix some $S_i$, and let $n$ be large enough so $f(S_i) \subset S_n$. Let $\{\alpha_k\}$ be the components of $\partial S_i \setminus \partial S$. First suppose $\alpha_k$ is a separating curve or arc. Since $S$ has at most one end accumulated by genus, $S \setminus \alpha_k$ has one component $U$ with finite genus. Increase $n$ if necessary so $S_n \cap U$ contains all of this genus. Note $f(\alpha_k)$ and $\alpha_k$ determine the same partition on $E(S_n)$ and the same partition on the compact boundary components of $S_n$.

Let $V = S_n \cap U$ and $W = S_n \cap f(U)$. Since $S_n$ contains all the genus of $U$, we must have that $V$ and $W$ have the same genus. It follows that $\alpha_k$ and $f(\alpha_k)$ have homeomorphic complementary domains. Now if $\alpha_k$ is nonseparating, then $\alpha_k$ and $f(\alpha_k)$ are both nonseparating in $S_n$, and so the complementary domains are homeomorphic in this case too. Therefore, we can find some $g \in \PMCG(S_n)$ which takes $f(\alpha_k)$ to $\alpha_k$. We can also require $gf$ to fix the orientation of $\alpha_k$ when it is a nonseparating curve.

Now we build a compactly supported element $f_i$ which approximates $f$ on $S_i$. Start by finding some $g_1 \in \PMCG(S_n)$ which takes $f(\alpha_1)$ to $\alpha_1$. Now find some $g_2 \in \PMCG(S_n \setminus \alpha_1)$ which takes $g_1f(\alpha_2)$ to $\alpha_2$. Repeat this process to find a sequence of compactly supported elements $\{g_k\}$ so $g = \cdots g_3g_2g_1$ sends each $f(\alpha_k)$ to $\alpha_k$. Also choose the $g_k$ so $gf$ fixes the orientation of each $\alpha_k$. Now finally we have that $g \in \PMCG_c(S)$ sends $f(S_i)$ to $S_i$. Then let $h_i$ be equal to $gf$ in $S_i$ and the identity outside $S_i$, so $f_i = g^{-1}h_i$ agrees with $f$ on $S_i$. 
\end{proof}

\subsection{Fragmentation}

The main tool for the proof of Theorem A is a fragmentation lemma that allows us to write a map in $\overline{\PMCG_c(S)}$ as a product of two simpler maps. This is based on fragmentation results from \cite{EK71} and \cite{Mann2019}, and was originally formulated by Domat in the case without boundary. Here we provide a proof that works in the general case.

\begin{lemma} [Fragmentation] \label{fraglemma}
    Let $S$ be any infinite type surface and $f \in \overline{\PMCG_c(S)}$. There exist two sequences of compact subsurfaces $\{K_{i}\}$ and $\{C_{i}\}$, with each sequence consisting of pairwise disjoint surfaces, and $g,h \in \overline{\PMCGc(S)}$ such that 
    \begin{enumerate}[(i)]
        \item
            $\supp(g) \subseteq \bigcup_{i} C_{i}$ and $\supp(h) \subseteq \bigcup_{i} K_{i}$,
        \item
            $f=hg$.
    \end{enumerate}
\end{lemma}

\begin{proof}
    Fix a compact exhaustion $\{S_i\}$ of $S$ by essential subsurfaces, and begin by setting $K_{1}' = S_1$. Choose some $n$ large enough so $f(K_1') \subset S_n$, and then set $K_1 = S_n$. Now there exists some $\phi_{1} \in \PMCG(K_{1})$ such that $\phi_{1}f$ fixes $\partial K_1'$. Let $$\psi_{1} = \phi_{1}f\vert_{K_{1}'} \in \PMCG(K_{1}')$$ Then $\psi_{1}^{-1}\phi_{1} \in \PMCG(K_{1})$ and $\psi_{1}^{-1}\phi_{1}f$ fixes $K_{1}'$. Let $g_{1} = \psi_{1}^{-1}\phi_{1}$. Next let $K_{2}',\ldots,K_{j}'$ be the components of some $S_n \setminus S_{n-1}$ where $n$ is large enough so that $f(K_i')$ is disjoint from $K_{1}$ for each $2 \leq i \leq j$.
    
    Now we run the same argument as before to get elements $\phi_{2},\ldots,\phi_{j}$ contained in some $\PMCG(K_{2}),\ldots, \PMCG(K_{j})$, respectively, with all of the $K_{i}$ pairwise disjoint such that $K^\prime_i \subseteq K_i$ and each $\phi_{i}f$ fixes $\partial K^\prime_i$. Our choices for the new $K_i$ will be the components of some $S_n \setminus S_m$ where $n$ and $m$ are any numbers such that $f(K^\prime_i) \subset S_n \setminus S_m$ for each $2 \leq i \leq j$, and $K_1 \subset S_m$. Then let $\psi_{i} = \phi_{i}f\vert_{K_{i}'}$ and $g_i = \psi_{i}^{-1}\phi_{i}$ so that each $g_{i}f$ fixes $K^\prime_i$.
  
    Continue this process to obtain an infinite sequence of elements $g_{i}$ and compact subsurfaces $K^\prime_i \subseteq K_{i}$ such that $g_i \in \PMCG(K_{i})$, each $g_if$ fixes $K^\prime_i$, and the $K_i$ are pairwise disjoint. The $g_{i}$ are compactly supported and have pairwise disjoint supports so the product $\cdots g_{3}g_{2}g_{1}$ converges to $\overline{g} \in \overline{\PMCGc(S)}$. Set $g = \overline{g}f$, so that $g$ fixes every $K^\prime_i$. Now let $\{C_i\}$ be the complementary domains of $\bigcup_i K^\prime_i$ in $S$, and note the $C_i$ are compact since each is contained in some $S_n \setminus S_m$. Note that in general the $C_i$ are allowed to intersect the $K_i$. Let $h = \overline{g}^{-1}$, so that $f = hg$. Now $\supp(h) \subseteq \bigcup_{i}K_{i}$ as desired. We also have $\bigcup_{i} C_{i} = S \setminus \bigcup_{i} K_{i}'$ and $g=\overline{g}f$ fixes each of the $K_{i}'$ which shows that $\supp(g) \subseteq \bigcup_{i} C_{i}$. 
    
    There is one subtlety we should mention. It will often be the case that a homeomorphism supported in some $K_i, C_i$ will be trivial in $\overline{\PMCGc(S)}$, so we should throw these subsurfaces out of our collections. For example, if the surface has any interior punctures then the $K_i, C_i$ will contain annuli bounding that puncture, and any map supported in the union of these annuli is trivial in $\overline{\PMCGc(S)}$. Note the above proof would also work if we were to instead work with the subgroup of $\Homeo^{+}_{\partial}(S)$ corresponding to homeomorphisms which can be approximated by compactly supported homeomorphisms. In this case, we would not throw out any of the subsurfaces. We could also relax the infinite type assumption if desired.
\end{proof}

\begin{figure}[ht]           
\begin{center}
\begin{tikzpicture}
\node[anchor=south west,inner sep=0] at (0,0) {\includegraphics[scale = 0.3]{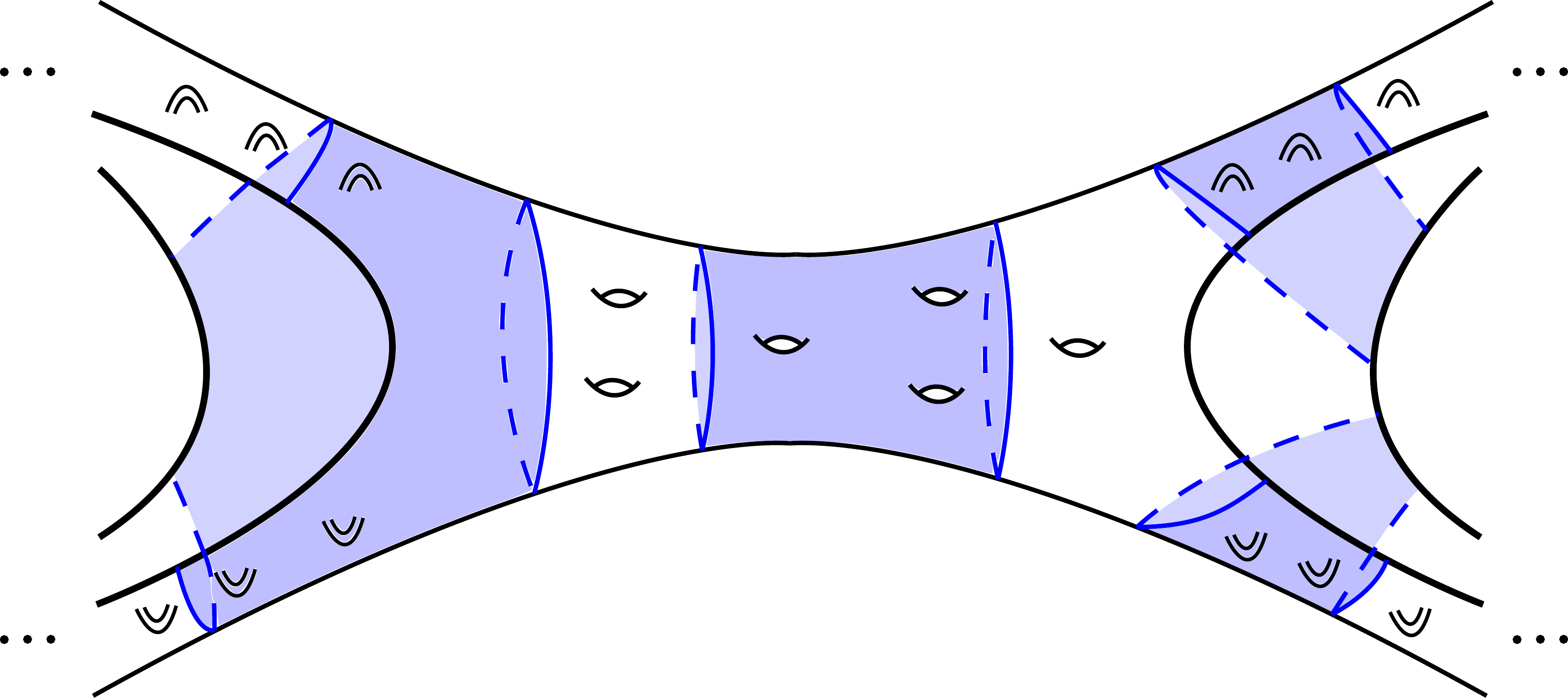}};

\end{tikzpicture}
\end{center}
\caption{Example of one of maps produced via fragmentation on a surface with two boundary chains (bold lines). The blue shaded regions represent the $K_i$ before we modify them.}
\label{frag}
\end{figure}

\begin{remark} \label{modify}
A critical observation is that some of the compact subsurfaces we get from fragmentation can be modified. Say $K$ is a compact subsurface whose boundary is composed of alternating essential arcs in $S$ and arcs in $\partial S$. Let $f \in \MCG(K)$, and conflate $f$ with a representative homeomorphism. Since $f$ fixes $\partial K$ we can assume after an isotopy that $f$ is the identity in an open regular neighborhood $N$ of $\partial K$, so $f \in \MCG(K^\prime)$ where $K^\prime = K \setminus N$. The boundary of $K^\prime$ is then a union of essential simple closed curves in $S$.
\end{remark}

Modifying the subsurfaces in this manner may turn a surface which separates into one that does not. For example, the rightmost two subsurfaces shown in Figure \ref{frag} can be modified to be nonseparating. This idea can be extended as follows. 

\begin{lemma} \label{fraglemma1}
Suppose $S$ is a Disk with Handles. Let $g$ and $h$ be maps given by fragmentation on some $f \in \PMCGcc{S}$, and let  $\{C_i\}$ and $\{K_i\}$ be the respective sequences of compact subsurfaces. We can assume the following: 

\begin{enumerate} [(i)]
\item Each $\partial K_i$ and $\partial C_i$ is a single essential simple closed curve.

\item $S \setminus \cup_i K_i$ and $S \setminus \cup_i C_i$ are homeomorphic to $S$ with compact boundary components added accumulating to some subset of the ends.
\end{enumerate} 

\end{lemma}

\begin{proof}
Recall fragmentation depends on a given choice of a compact exhaustion $\{S_i\}$. By Proposition 4.10, we can choose our exhaustion so each $\partial S_i$ has one component composed of alternating essential arcs in $S$ and arcs in $\partial S$. From the proof of Lemma \ref{fraglemma}, we have that each $C_i, K_i$ is either some $S_n$ or a component of some $S_n \setminus S_m$. We now show we can assume the desired conditions for $h$ and the $K_i$, and the proof for the other map is similar. Since each component of $\partial K_i$ intersects $\partial S$ we can modify the $K_i$ as in Remark 5.4 so $\partial K_i$ is a union of essential simple closed curves. See Figure \ref{fragloch} for an example. In the case of fragmentation on the 2-Sliced Loch Ness Monster (see Figure \ref{doublysliced} and Construction \ref{construction}), this process will often give $K_i$ with two boundary components, and in general this can give any finite number of boundary components.
 
Now note that none of the $K_i$ bound a common subsurface. By selecting the compact subsurfaces carefully in the proof of Lemma \ref{fraglemma}, we can assume that $S \setminus \cup_i K_i$ has infinite genus. It follows that $S \setminus \cup_i K_i$ is a Disk with Handles with compact boundary components added.  We can also assume an end in $S \setminus \cup_i K_i$ is accumulated by genus if and only if the corresponding end of $S$ is accumulated by genus. Therefore, we have the second condition of the lemma.

For any $K_i$ with $n$ compact boundary components where $n >1$, connect all the components together with $n-1$ disjoint arcs $\{\alpha_k\}_{k=1}^{n-1}$ in $S \setminus \cup_i K_i$. Now enlarge $K_i$ by adding in a small closed regular neighborhood of $\partial K_i \cup \bigcup_k \alpha_k$. Repeat this for every $K_i$ making sure the new subsurfaces are all disjoint. Now we have the first condition of the lemma. In order to maintain the second condition, we must also assume that only finitely many of the new $K_i$ intersect any given compact subsurface. This is possible by choosing the arcs at each stage carefully. In particular, at each step let $S_{j_i}$ be the largest subsurface in the original compact exhaustion which does not intersect $K_i$ and choose the arcs to be outside of $S_{j_i}$. 
\end{proof}

\begin{figure}[htpb!]
\begin{center}
\resizebox{.83\textwidth}{!}{
\begin{tikzpicture}
\node[anchor=south west,inner sep=0] at (0,0) {\includegraphics[scale = 0.45]{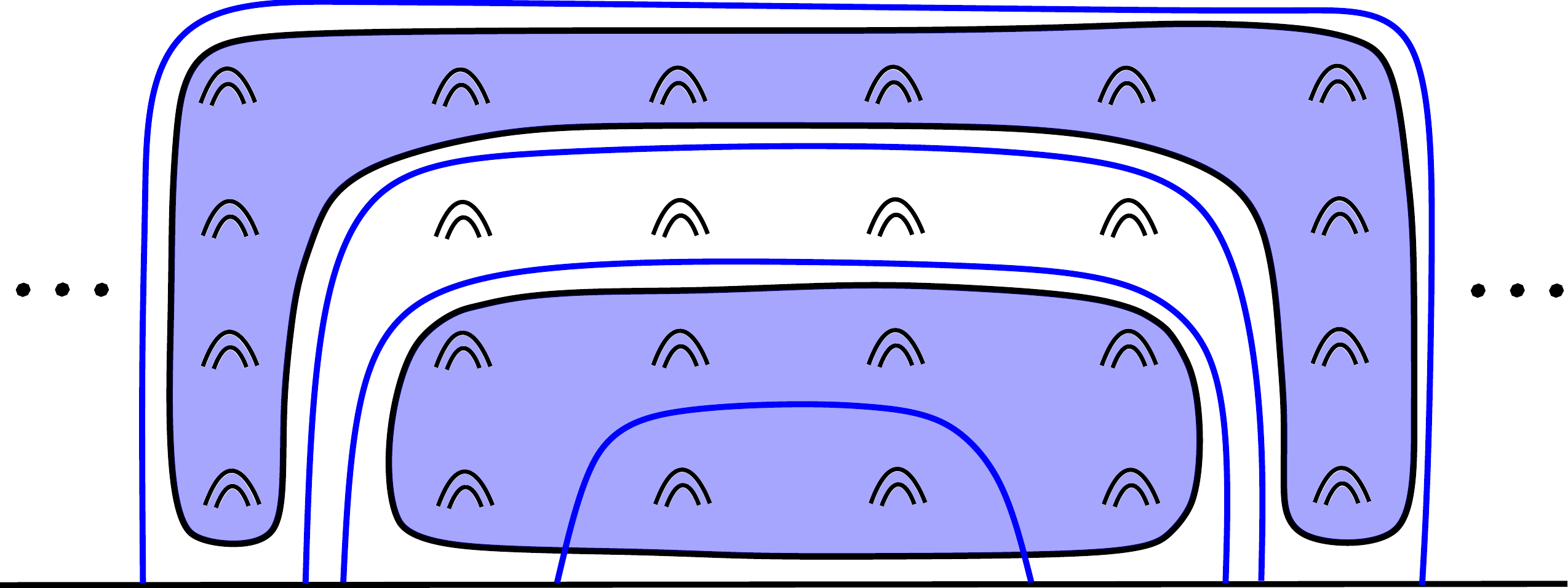}};

\node[anchor=south west,inner sep=0] at (4.1,1.5) {$K_1$};
\node[anchor=south west,inner sep=0] at (2.4,3.5) {$K_2$};

\end{tikzpicture}}
\end{center}
\caption{Example of fragmentation on a 1-Sliced Loch Ness. The blue arcs correspond to the compact exhaustion used for the fragmentation. The $K_i$ correspond to the modified subsurfaces containing the support of one of the maps from fragmentation.}
\label{fragloch}
\end{figure}

\begin{lemma} \label{fraglemma2}
    Suppose $S$ is a connected sum of finitely many Disks with Handles. Let $g$ and $h$ be maps given by fragmentation on some $f \in \PMCGcc{S}$, and let  $\{C_i\}$ and $\{K_i\}$ be the respective sequences of compact subsurfaces. We can assume the following:
    \begin{enumerate} [(i)]
        \item $S \setminus K_1$ and $S \setminus C_1$ are Disks with Handles with one compact boundary component added. 
        \item For the remaining $C_i, K_i$, each $\partial K_i$ and $\partial C_i$ is a single essential simple closed curve.
        
        \item Each component of $S \setminus \cup_i K_i$ and $S \setminus \cup_i C_i$ is a Disk with Handles with compact boundary components added accumulating to some subset of the ends.
    \end{enumerate}
\end{lemma}

\begin{proof}
Suppose $S$ is a connected sum of $n$ Disks with Handles. By piecing together compact exhaustions of the Disks with Handles and using Proposition 4.10, we can choose our exhaustion $\{S_i\}$ of $S$ for fragmentation so each $\partial S_i$ has $n$ components, each corresponding to one of the boundary chains, composed of alternating essential arcs in $S$ and arcs in $\partial S$. For the $h$ map, $K_1$ is equal to some $S_n$. Then modifying $K_1$ as in Remark \ref{modify} gives the first condition of the lemma. We get the remaining conditions for this map by following the proof of Lemma \ref{fraglemma1} for each component of $S \setminus K_1$. For the $g$ map, enlarge its $C_1$ to be some $S_n$ which contains $K_1$ and any of the $C_i$ which intersect $K_1$. Then we get the desired conditions for this map by the same argument. Note we could have stated a version of this lemma with different conditions for this second map, but that will not be necessary for the following proofs.
\end{proof}

\subsection{Proof of Theorem A}
First we use fragmentation along with standard commutator tricks to show every element of $\PMCGcc{S}$ can be written as a product of two commutators when $S$ is a Sliced Loch Ness Monster. Then we will show the same for any Disk with Handles by applying Lemma \ref{diskdecomp}. Finally we extend to the remaining cases using Lemma \ref{cutchains}. During the upcoming proofs, we are implicitly using the fact that $$\PMCGcc{S} = \PMCG(S) = \MCG(S)$$ when $S$ is a Sliced Loch Ness Monster.

\begin{lemma} \label{diskperfect}
    $\PMCGcc{S}$ is uniformly perfect when $S$ is a Disk with Handles.
\end{lemma}

\begin{proof}
Let $g$ be any of the two maps given by fragmentation on a general $f \in \PMCGcc{S}$, and let $\{C_i\}$ be the corresponding sequence of compact subsurfaces. First consider the case when our surface is the 1-Sliced Loch Ness Monster, $L_s$. By Lemma \ref{fraglemma1} we may assume each $\partial C_i$ has one component and the complement of $\cup_i C_i$ is homeomorphic to $L_s$ with infinitely many compact boundary components added accumulating to the single end.

Realize $L_s$ as the closed upper half plane with a handle attached inside an $\epsilon$-ball at every integer point, and let $\psi$ be the map $(x,y) \to (x+1, y)$ extended to the attached handles and isotoped in a neighborhood of the boundary to be the identity. Now we can assume using the change of coordinates principle or by replacing $g$ with a conjugate that the $C_i$ are contained inside the vertical strip bounded by the lines $x = \pm \frac{1}{2}$ and also the support of $\psi$. Letting $a = \Pi_{k\geq0}\psi^{k}g\psi^{-k}$ we can now write $$g = \psi a^{-1}\psi^{-1}a = [\psi, a^{-1}]$$ See Figure \ref{ls} for a visualization. It now follows that we can write any $f \in \PMCGcc{S}$ as the product of two commutators. 

\begin{figure}[ht]
\begin{center}
\resizebox{.83\textwidth}{!}{
\begin{tikzpicture}
\node[anchor=south west,inner sep=0] at (0,0) {\includegraphics[scale = 0.35]{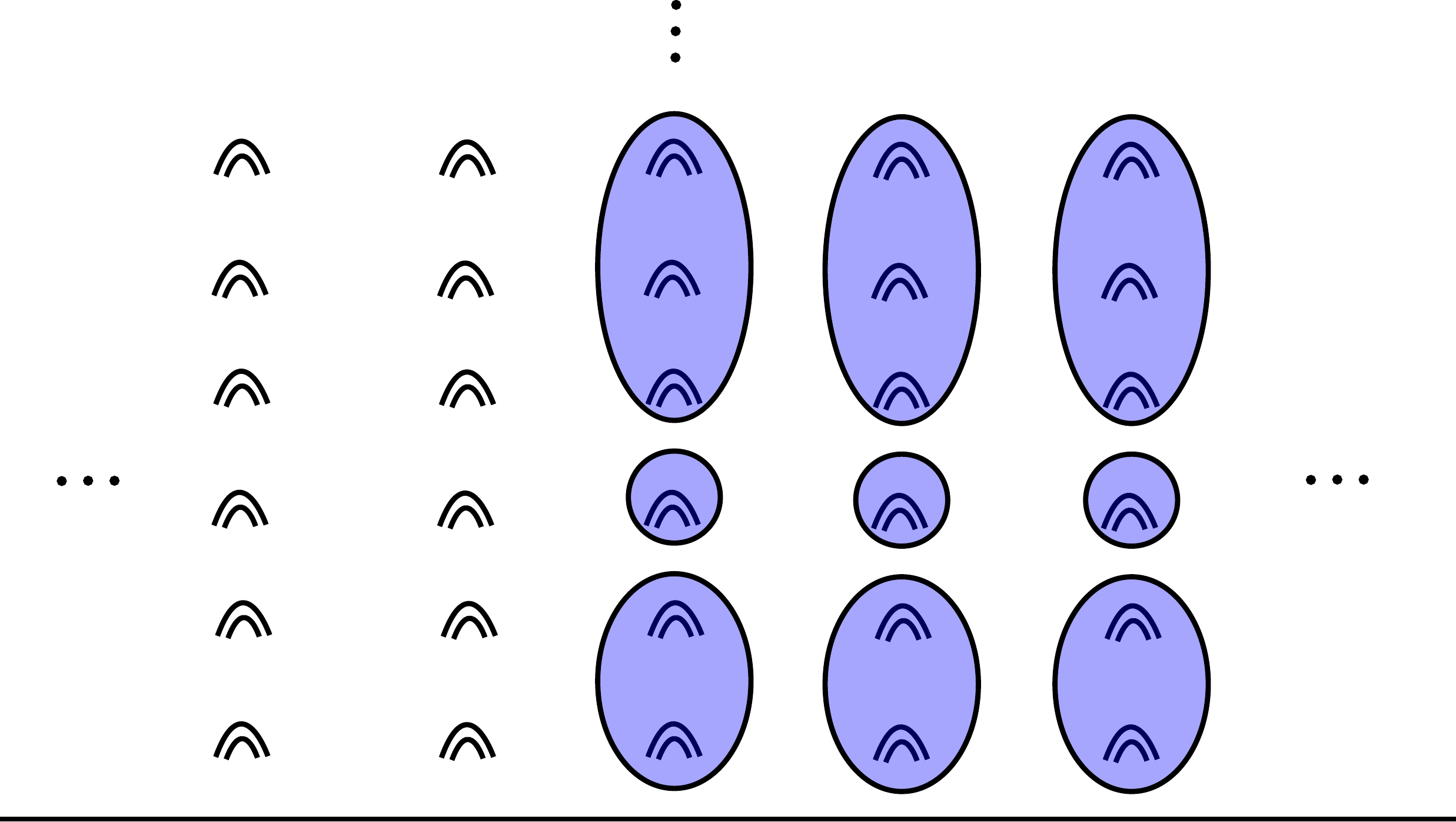}};

\node[anchor=south west,inner sep=0] at (3.15,1.2) {\small $C_1$};
\node[anchor=south west,inner sep=0] at (3.15,2) {\small $C_2$};
\node[anchor=south west,inner sep=0] at (3.15,4) {\small $C_3$};

\end{tikzpicture}}
\end{center}
\caption{The last step for showing $\PMCGcc{L_s}$ is uniformly perfect. The support of the element $a$ is shown in blue.}
\label{ls}
\end{figure}
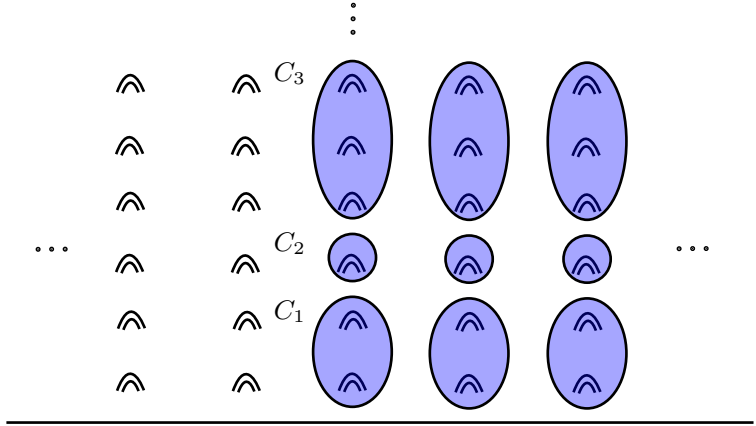

Next we extend this to any $n$-Sliced Loch Ness Monster, $L^n_s$. First we need a model of this surface that works with the above method. Take a copy of $L_s$ with the above half plane model, and denote it by $T$. Now take the disjoint union with $n-1$ new copies of $L_s$ realized in any way. Attach handles from $T$ to each additional copy of $L_s$ to join all the ends into a single end. When removing open balls from $T$ in the process of attaching these handles, choose the open balls to be below the line $y = \frac{1}{2}$. Similar to Construction \ref{construction}, this yields a surface homeomorphic to $L^n_s$ which we use as our model. See Figure \ref{3sliced} for an example when $n=3$. Let $T^\prime \subset L^n_s$ be the subsurface corresponding to the area of $T$ above the line $y = \frac{1}{2}$. Now we can use Lemma \ref{fraglemma1} and the change of coordinates principle as before to assume that the $C_i$ are contained above the attached handles within a vertical strip of $T^\prime$. We then let $\psi$ be the map which acts as the previous shift map on $T^\prime$ and fixes the remainder of $L^n_s$. Proceed as before to show $g = [\psi, a^{-1}].$

\begin{figure}[ht]
\begin{center}
\begin{tikzpicture}
\node[anchor=south west,inner sep=0] at (0,0) {\includegraphics[scale = 0.25]{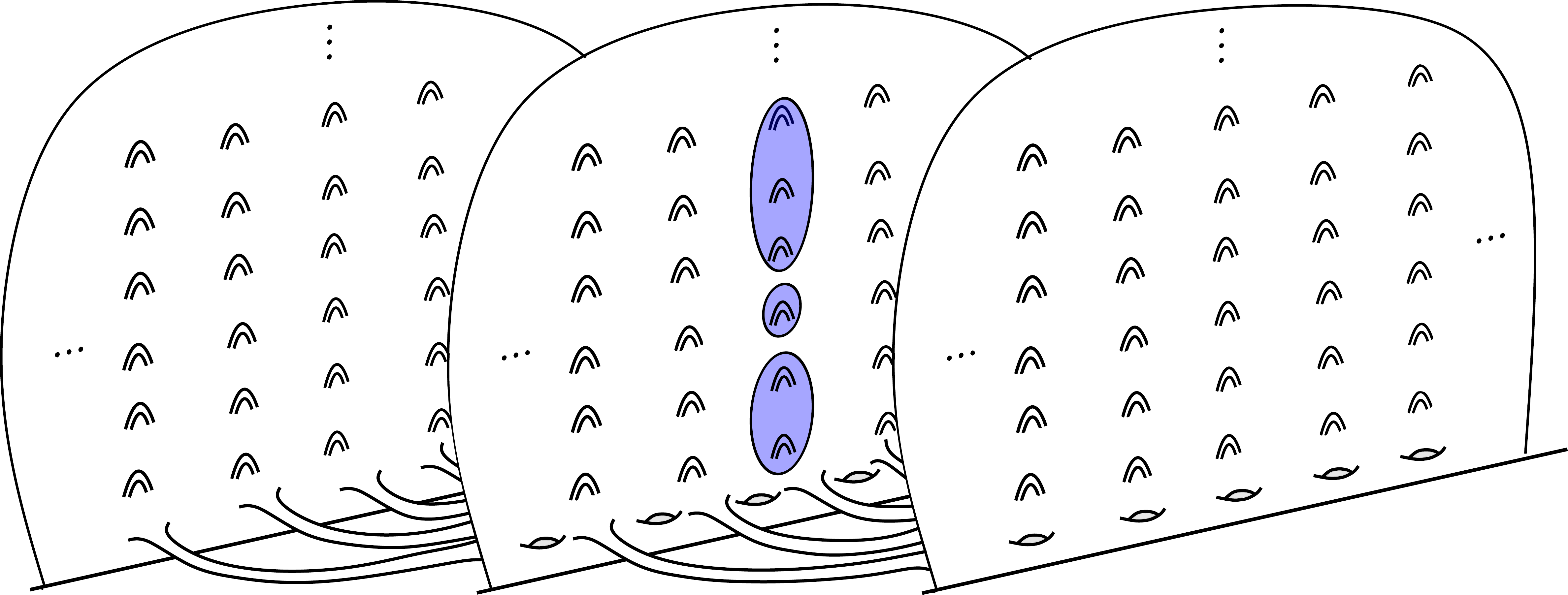}};

\end{tikzpicture}
\end{center}
\caption{A model of the 3-Sliced Loch Ness Monster used in the proof of Lemma \ref{diskperfect} with the surfaces $C_i$ shown in blue in a vertical strip in the middle piece.}
\label{3sliced}
\end{figure}

Now suppose $S$ is any Disk with Handles. After applying Lemma \ref{fraglemma1}, we can assume $S \setminus \cup_i C_i$ is homeomorphic to $S$ with infinitely many compact boundary components added accumulating to some subset of the ends. Using a slight variation of Lemma \ref{diskdecomp} where we allow the Disks with Handles to have compact boundary components added, we can cut $S$ along a collection of disjoint arcs $\{\alpha_j\}$ which miss the $C_i$ so the components of the cut surface are Sliced Loch Ness Monsters. When we then cut out the $C_i$ we get Sliced Loch Ness Monsters with compact boundary components added. Give the components the models discussed in the previous paragraphs, and apply the change of coordinates principle argument to each component to assume each $C_i$ is contained within a vertical strip within its respective component. Let $\{\psi_i\}$ be the collection of plane shift maps for each component analogous to the previous paragraphs. Since the supports of the $\psi_i$ are disjoint we have a well-defined product $\psi = \Pi_{i} \psi_i$, and then we can show $g = [\psi, a^{-1}]$ as before.

\end{proof}

\begin{lemma} \label{diskperfectextended}
    $\PMCGcc{S}$ is perfect when $S$ is a connected sum of finitely many Disks with Handles with possibly finitely many punctures or compact boundary components added.
\end{lemma}

\begin{proof}
Let $g$ be a map given by fragmentation on a general $f \in \overline{\PMCG_c(S)}$, and let $\{C_i\}$ be the corresponding compact subsurfaces. First suppose $S$ has no punctures or compact boundary components. When fragmenting in this case, we get supports with boundary components that are curves which separate ends (see the two leftmost subsurfaces from Figure \ref{frag}). If a map is supported within one of these subsurfaces, then we cannot move the support off of itself as we did in the other cases. This is commonly referred to as a nondisplaceable subsurface (see Definition 1.8 of \cite{mann2019large}).

Apply Lemma \ref{fraglemma2} so we can assume the $C_i$ have the desired properties. We can assume $C_1$ has genus at least 3 by replacing it with a connected compact surface containing $C_1$ and more of the $C_i$. Now $g = g_1g_2$ where $g_1 \in \PMCG(C_1)$ and $g_2 \in \PMCG(S \setminus C_1).$ The classic result of Powell \cite{Pow1978} tells us we can write $g_1$ as a product of commutators. By the method in the previous lemma, we can write $g_2$ as a single commutator. It follows every element in $\overline{\PMCG_c(S)}$ can be written as a product of commutators.

The cases with finitely many punctures and compact boundary components are done similarly. To consider the cases with punctures, we can slightly modify the fragmentation process by replacing a compact exhaustion with an exhaustion of finite type surfaces. Then depending on the number of boundary chains, we use a modification of either Lemma \ref{fraglemma1} or Lemma \ref{fraglemma2} such that $C_1$ includes the boundary components and punctures. 

\end{proof}

These lemmas complete the reverse implications from Theorem A, so now we discuss why the other directions hold. For all infinite type surfaces with only compact boundary components, $\overline{\PMCGc(S)}$ is not perfect by the work of Domat. His proof relies on finding a particular sequence of disjoint essential annuli. Then he shows some multitwist about the core curves of these annuli cannot be written as a product of commutators. His work can be summarized with the following theorem. For the statement of this theorem, a nondisplaceable surface in $S$ refers to an essential subsurface $K$ disjoint from the noncompact boundary components of $S$ such that $f(K) \cap K \ne \emptyset$ for all $f \in \PMCGcc{S}$. Note a subsurface $K$ is nondisplaceable if it separates ends; i.e., if $S \setminus K$ is disconnected and induces a partition of $E(S)$ into two sets. A subsurface is also nondisplaceable if some component of $S \setminus K$ is a finite type subsurface containing a compact boundary component of $\partial S$.

\begin{theorem} \label{domat} \cite{domat2020big}
 Let $S$ be an infinite type surface such that there exists an infinite sequence of disjoint nondisplacable essential annuli that eventually leaves every compact subsurface. Then $\PMCGcc{S}$ is not perfect. 
\end{theorem}

The hypothesis of Theorem \ref{domat} holds whenever there are interior ends of $S$ accumulated by genus, except in the case of the Loch Ness Monster which was handled separately in the appendix of \cite{domat2020big}. It also holds if there are infinitely many planar ends or infinitely many compact boundary components. By using Lemma \ref{cutchains}, we see the hypothesis of Theorem \ref{domat} holds whenever there are infinitely many boundary chains as well. The only cases that remain are exactly the surfaces from Lemma \ref{diskperfect} and Lemma \ref{diskperfectextended}. This proves the forward direction of the second bullet point in Theorem A.
  
Finally in order to show the forward direction of the first bullet point, we must explain why $\overline{\PMCGc(S)}$ is not uniformly perfect when $S$ has more than one boundary chain, any planar ends, or any compact boundary components. We will only sketch the details since the main ideas here are taken from \cite{domat2020big}. The issue in these cases is that there is some essential curve $\alpha$ which is nondisplaceable under the action of $\PMCGcc{S}$. Either take a curve which separates ends or bounds a finite type subsurface containing a compact component of $\partial S$. The orbit of $\alpha$ can then be used to build a Bestvina-Bromberg-Fujiwara projection complex (see \cite{BBF2019}) on which $\overline{\PMCGc(S)}$ acts by isometries. This complex is quasi-isometric to a tree, and the Dehn twist about $\alpha$ is a WWPD element. An adaptation of a construction of Brooks \cite{Bro1981} then gives a quasimorphism from $\overline{\PMCGc(S)}$ to $\R$ which is unbounded on $\{T_\alpha^n\}_{n=1}^\infty$. Combining this with the fact that homogeneous quasimorphisms are bounded on commutators, we see $\overline{\PMCGc(S)}$ cannot be uniformly perfect.


\section{Extending Results} \label{APV}

\subsection{Background}
In the case of surfaces with only compact boundary components, it is known that $\PMCG(S)$ factors as a semidirect product containing $\overline{\PMCG_c(S)}$ as one of the factors.

\begin{theorem} [Aramayona-Patel-Vlamis \cite{APV2017}, Corollary 6] \label{structure}
 Let $S$ be an infinite type surface with compact boundary components. Then $$\PMCG(S) = \overline{\PMCG_c(S)} \rtimes H$$ where $\displaystyle H \cong \Z^{n-1}$ when there is a finite number $n > 1$ of ends of $S$ accumulated by genus, $\displaystyle H \cong \Z^\infty$ when there are infinitely many ends accumulated by genus, and $H$ trivial otherwise. Furthermore, $H$ is generated by pairwise commuting handle shifts.  
\end{theorem}

Here $\Z^\infty$ refers to the direct product of a countably infinite number of copies of $\Z$. Although many of the results of Aramayona-Patel-Vlamis are stated for the case of compact boundary, the proofs all apply to surfaces with only compact boundary components.

In order to extend this result, we will also need to extend a well-known fact about when the inclusion of a subsurface induces an injective map between mapping class groups. Recall the definition of a degenerate end (Definition \ref{degenerate}). 
We say a boundary chain of a surface is degenerate when every end in the chain is degenerate. After filling in degenerate ends, degenerate chains become compact boundary components. Similar to a Dehn twist about a compact boundary component, we can also speak of a Dehn twist about a degenerate chain. 

\begin{lemma} \label{injective}
Let $S$ be any surface, and $\Sigma$ a closed essential subsurface. The natural homomorphism $i: \PMCG(\Sigma) \rightarrow \PMCG(S)$ is injective when the following holds:

\begin{enumerate}[(i)]
    \item No compact component of $\partial \Sigma$ bounds a disk with a single interior puncture.
    \item No two compact components of $\partial \Sigma$ bound an annulus. 
    \item There are no degenerate chains in $\Sigma$ such that each boundary component of the chain bounds an upper half plane. 

\end{enumerate}
\end{lemma}

The proof will rely on the Alexander method for infinite type surfaces. The case of compact boundary components was done in \cite{HMV2019}. We will use a slight modification of the standard definition for a stable Alexander system.

\begin{definition}
A stable Alexander system for a surface without degenerate ends is a locally finite collection of essential simple closed curves and essential arcs $\Gamma$ in a surface $S$ such that the following properties hold:

\begin{itemize}
    \item The elements in $\Gamma$ are in pairwise minimal position.
    \item For distinct $\alpha_i, \alpha_j \in \Gamma$, we have that $\alpha_i$ is not isotopic to $\alpha_j$.
   \item For all distinct $\alpha_i, \alpha_j, \alpha_k \in \Gamma$, at least one of the following sets is empty: $\alpha_i \cap \alpha_j$, $\alpha_i \cap \alpha_k$, $\alpha_j \cap \alpha_k$.
    \item The collection $\Gamma$ fills $S$; i.e., each complementary component is a disk or a disk with a single interior puncture.
   \item Every $f \in \Homeo^+_{\partial}(S)$ that preserves the isotopy class of each element of $\Gamma$ is isotopic to the identity.
\end{itemize}

We say $\Gamma$ is a stable Alexander system for a surface with degenerate ends if it becomes a stable Alexander system when the degenerate ends are filled in.
\end{definition}

\begin{theorem} (Alexander method)
For any infinite type surface $S$, there exists a stable Alexander system $\Gamma$.
\end{theorem}

\begin{proof}
We will assume the compact boundary case from \cite{HMV2019}. First suppose $S$ has noncompact boundary and no degenerate ends. Embed it in the natural way inside the double, $dS$. Let $\Gamma$ be a stable Alexander system for $dS$. 

For an arbitrary $\gamma \in \Gamma$, isotope it to be transverse and in minimal position with $\partial S$ so $\gamma \cap S$ is either a curve or a union of arcs in $S$. Let $\Gamma^\prime$ be the collection of all curves and arcs formed in this manner. After possibly removing repeated occurrences of isotopy classes, $\Gamma^\prime$ is a stable Alexander system for $S$. 

Now suppose $S$ has degenerate ends. Apply the argument to $S$ with the degenerate ends filled in, then isotope the arcs along the boundary if necessary so they descend to arcs in $S$. 
\end{proof}

The proof of Lemma \ref{injective} will also rely on some facts about arcs. Note given the current definition of an essential arc, in a surface with degenerate ends there may be essential arcs which bound a disk with boundary points removed. These arcs can be isotoped to be disjoint from any curve. In fact, we have the following. 

\begin{proposition} \label{degeneratearcs}
  Let $S$ be a surface which contains essential simple closed curves. An essential arc $\alpha$ in $S$ can be isotoped to be disjoint from any curve if and only if it bounds a disk with boundary points removed. 
\end{proposition} 

\begin{proof}
The reverse direction is clear, so suppose some essential arc $\alpha$ can be isotoped to be disjoint from any curve. Let $\{S_i\}$ be an exhaustion of $S$ by compact essential subsurfaces. For any $S_i$ large enough to contain $\alpha$, we must have that $\alpha$ bounds a disk in $S_i$. Otherwise, we could construct a curve in $S$ which cannot be isotoped away from $\alpha$. It follows $\alpha$ is separating and that a component of $S \setminus \alpha$ has a compact exhaustion composed of only disks. This component cannot be compact since then $\alpha$ would be trivial, and it cannot contain compact boundary components or interior ends since then we can construct a curve which cannot be isotoped away from $\alpha$. By Proposition 4.10, this component has a single boundary chain. Since it has no genus, no compact boundary components, and no interior ends it follows that it must be a disk with boundary points removed.  
\end{proof}

For the following proposition and its proof, we allow all isotopies of arcs to move the endpoints along the boundary. 

\begin{proposition}
  \label{curvelemma}
 Let $S$ be an infinite type surface with nonempty boundary, and $\alpha$ an essential arc in $S$. There exists a collection of curves $\Gamma$ disjoint from $\alpha$ so the following holds: If $\beta$ is an arc with endpoints on the same boundary components as $\alpha$, and $\beta$ can be isotoped to be disjoint from any curve in $\Gamma$, then $\alpha$ and $\beta$ are isotopic.

\end{proposition}

\begin{proof}
First suppose $S$ has no degenerate ends. Let $\{S_i\}$ be a compact exhaustion of $S$. Delete the first few subsurfaces in the exhaustion so each $S_i \setminus \alpha$ is complex enough to contain essential simple closed curves. First suppose $\alpha$ is nonseparating in $S_i$. Then let $\Gamma_i$ be a finite collection of curves in minimal position which fills the interior of $S_i \setminus \alpha$, so each complementary component of $\Gamma_i$ in $S_i$ is a disk or an annulus. When $\alpha$ is separating in $S_i$, it is possible it bounds an annulus or a pair of pants. Then let $\Gamma_i$ be a collection which fills the interior of the other component. If the compact component is a pair of pants, add the curve bounding the two boundary components not containing $\alpha$ to $\Gamma_i$. For all other cases, we just let $\Gamma_i$ be a collection which fills the interiors of both components of $S_i \setminus \alpha$. Let $\Gamma = \bigcup_i \Gamma_i$.

Suppose $\beta$ is any arc as in the statement of the lemma. Choose some $i$ large enough so $S_i$ contains both $\alpha$ and $\beta$, and so that $\alpha$ and $\beta$ have endpoints on the same boundary components of $S_i$. Now isotope $\beta$ to be disjoint from every curve in $\Gamma_i$. Let $A$ be the complementary component of $\Gamma_i$ in $S_i$ which contains $\alpha$ and $\beta$. Note that the complementary components of $\Gamma_i$ in $S \setminus \alpha$ which intersect $\alpha$ are annuli. Therefore, $A$ is the result of gluing two annuli together along a pair of arcs on their boundaries or by gluing a single annulus to itself along two arcs on the boundary. These arcs all correspond to $\alpha$ in $A$ after the gluing. The single annulus case only occurs when $\alpha$ is an arc between two different compact boundary components of $S_i$, and in this case the annulus gets glued to itself by arcs on the same boundary component. It follows that $A$ is a pair of pants. It is standard fact that there is a unique arc, up to isotopy, between any two boundary components of a pair of pants (see \cite{primer} Proposition 2.2). It follows $\beta$ must be isotopic to $\alpha$ in $S_i$. Since this holds for all sufficiently large $i$, we see that $\beta$ is isotopic to $\alpha$ in $S$.

Now suppose the surface has degenerate ends. If $\alpha$ does not become trivial after these ends are filled in, then we can apply the above argument to the filled in surface to get the desired collection of curves. Otherwise, let $\Gamma$ be the collection of all curves in $S$. By Proposition  \ref{degeneratearcs}, if $\beta$ can be isotoped to be disjoint from every curve it must be an arc which bounds a disk with boundary points removed. The arc $\alpha$ also has this property. Now since $\alpha$ and $\beta$ have endpoints on the same boundary components, they induce the same partition of the ends space and have homeomorphic complementary components so it follows that $\alpha$ and $\beta$ are isotopic. 
\end{proof}

\begin{proof} [Proof of Lemma \ref{injective}]
The last condition is similar to the first condition in the sense that it prevents Dehn twists from being in the kernel; in this case, Dehn twists about degenerate chains. For example, consider any compact surface with one boundary component and then delete an embedded closed subset of the Cantor set from the boundary to form a degenerate chain. Attaching closed upper half planes to each boundary component in the degenerate chain yields a surface with a single puncture, and the Dehn twist about the chain becomes trivial in the mapping class group of the new surface. We give a proof following Farb-Margalit \cite{primer}.

Let $f \in \PMCG(\Sigma)$ be in the kernel, and conflate it with a representative homeomorphism. We extend $f$ by the identity to a homeomorphism which represents $i(f)$. Let $\Gamma$ be a stable Alexander system for $\Sigma$.

Let $\alpha$ be any essential simple closed curve in $\Sigma$. Since $i(f)$ is isotopic to the identity and $i(f)$ agrees with $f$ on $\Sigma$, we have that $f(\alpha)$ is isotopic to $\alpha$ in $S$. Let $K \subset S$ be a compact essential subsurface which contains this isotopy. If $K$ can be isotoped to be contained within $\Sigma$ then we are done, so assume otherwise. Now after isotoping $\partial K$ and $\partial \Sigma$ to be transverse and in minimal position, $K \cap \partial \Sigma$ is a union of arcs in $K$. Since $f(\alpha)$ and $\alpha$ are contained in the interior of $\Sigma$, they are disjoint from these arcs, and it follows from a standard fact of isotopies in the compact case that there is an isotopy in $\Sigma$ from $f(\alpha)$ to $\alpha$ missing the arcs. See for example \cite{primer} Lemma 3.16. Although the lemma here is stated for curves instead of arcs, the same proof extends to our setting with minor changes. Therefore, we have that $f$ fixes the isotopy class of every curve in $\Sigma$.

 Let $\alpha$ be an arbitrary arc in $\Gamma$. By Proposition \ref{curvelemma}, we can find a collection of curves in $\Sigma$ such that $f(\alpha)$ is isotopic to $\alpha$, by an isotopy possibly moving the endpoints, if it can be isotoped to miss each curve in the collection. This last condition holds since $f$ fixes the isotopy class of every curve. Now we can assume by an isotopy not moving the endpoints that $f(\alpha)$ agrees with $\alpha$ outside of an open collar neighborhood $N$ of the boundary components. Since $\Gamma$ descends to a stable Alexander system for $S \setminus N$, we can apply the Alexander method to $S \setminus N$ to show $f$ is supported in $N$. The components of $N$ are annuli and strips $\R \times [-1,1]$. Since the mapping class groups of the latter components are trivial, it follows that $f$ is a possibly infinite product of Dehn twists supported in the annuli. By the given conditions, we must now have that $f$ is isotopic to the identity, since otherwise $i(f)$ would be nontrivial. 

\end{proof}

\begin{remark}
Note deleting a noncompact boundary component is topologically the same as attaching an upper half plane to the component. Therefore, we can extend the above proof to show that homomorphisms such as the one discussed in Section \ref{outline} are injective. In particular, we will still have injectivity as long as we do not delete any degenerate chains or compact boundary components.
\end{remark}

As an application of Lemma \ref{injective}, we mention a potentially useful theorem.

\begin{theorem}
Let $S$ be an infinite type surface with no compact boundary components and no degenerate chains, and suppose $f \in \MCG(S)$ fixes the isotopy class of every curve. Then $f$ must be the identity.
\end{theorem}

\begin{proof}
The conditions on $S$ are necessary since otherwise a Dehn twist about a compact boundary component or degenerate chain would provide a counterexample. 

Let $S^\prime = S \cup_{\partial S} (\partial S \times [0, \infty))$, and let $i$ be the map from $\MCG(S)$ to $\MCG(S^\prime)$ induced by the inclusion of $S$ into $S^\prime$. Since the conditions of Lemma \ref{injective} are satisfied by this inclusion, $i$ must be injective. Curves in $S^\prime$ can always be isotoped by an innermost bigon argument to be inside of $S$, so therefore we have that $i(f)$ must fix every curve in $S^\prime$ up to isotopy. By the Alexander method for surfaces without boundary, $i(f)$ must be the identity, and so $f$ must be as well by injectivity of $i$.
\end{proof}

Now we will also prove a theorem which is a direct extension to the result shown in Section 2.

\begin{theorem}
 Let $S$ be an infinite type surface with at least one nondegenerate boundary chain. Then the map $i: \MCG(S) \rightarrow \MCG(S^\lilo)$ given by restricting a mapping class to the interior is not surjective.
\end{theorem}

\begin{proof}
 By Lemma \ref{cutchains} we can cut $S$ along curves so that each component of the cut surface has at most one boundary chain. Consider one of the components $A$ which has a nondegenerate boundary chain. By Lemma \ref{cutchains} we can assume $A$ has no interior ends. Now $A$ must have a boundary end which is either accumulated by genus or compact boundary components. Cap all the compact boundary components with disks, and then apply Proposition $\ref{singlechainequiv}$ to get a compact exhaustion $\{A_i\}$ of $A$ such that each $\partial A_i$ has one component. Isotope each $\partial A_i$ into the interior of $A$ to get a curve $\alpha_i$. Note we can assume after isotopies that $\{\alpha_i\}$ is a pairwise disjoint collection and each $\alpha_i$ is disjoint from the disks used to cap the compact boundary components.
 
 Undo the capping of the compact boundary components, and then note each $\alpha_i$ bounds a compact subsurface and these subsurfaces form a compact exhaustion of $A \setminus C$, where $C$ is the union of noncompact boundary components of $A$. Observe that $\{\alpha_i\}$ contains infinitely many nonisotopic curves. Otherwise, $\alpha_{i+1}$ and $\alpha_i$ would bound an annulus for all sufficiently large $i$. Then by considering the compact exhaustion of $A \setminus C$ given by the $\alpha_i$, we see that $A \setminus C$, and therefore $A$, has finite genus and finitely many compact boundary components. However, this is not possible by assumption. Now throw away any repeated occurrences of isotopy classes from $\{\alpha_i\}$. 

Now we want to show that $T = \Pi_{i=1}^{\infty}T_{\alpha_i} \in \MCG(S^\lilo)$ is not in the image of $i$. Let $\gamma$ be any essential arc in $A \subseteq S$ with endpoints on the noncompact boundary components such that $\gamma$ does not bound a disk with boundary points removed. Now we use the same approach from Section 2 to show that if $T$ were in the image of $i$ then there would be a homeomorphism on $S$ which sends $\gamma$ to something noncompact, a contradiction. Conflate $T$ with a homeomorphism on $S$ which restricts to $T$ on the interior. By the construction of the $\alpha_i$ and $\gamma$, for all sufficiently large $i$ we have that $\gamma$ cannot be isotoped to be disjoint from $\alpha_i$. Note here we are implicitly using Proposition 6.5 applied to $\gamma$. Now we can find an infinite collection of curves $\{ \beta_i \}$ which eventually leaves every compact subsurface of $S$ such that each $\alpha_i$ intersects $\beta_i$, and therefore $T(\gamma)$ intersects each $\beta_i$. We are then done since it follows that $T(\gamma)$ is noncompact. One approach for finding the $\beta_i$ is to consider a compact exhaustion $\{S_i\}$ of $S$ and choose each $\beta_i$ in some $S_{n_i} \setminus S_{m_i}$ where $n_i$ and $m_i$ go to infinity as $i$ does.  
\end{proof}

\subsection{Extending Aramayona-Patel-Vlamis}

First we will give a proof of Theorem \ref{structure}, and then explain how to extend it to the general case. We say a handle shift $h$ cuts a curve $\alpha$ when $h^+$ and $h^-$ are on opposite sides of $\alpha$. Let $S$ be a surface with only compact boundary components. A \textit{principal exhaustion} of $S$ is an exhaustion of $S$ by finite type subsurfaces such that the following conditions hold for all $i$:
   \begin{enumerate}[(i)]
        \item 
            Each complementary domain of $S_{i}$ is an infinite type surface.
        \item
            Each component of $\partial S_{i}$ is separating.  
   \end{enumerate}
Now we state a few results from \cite{APV2017} which we will assume for the following proofs. Let $H_1^{sep}(S, \Z)$ denote the subgroup of the first homology of a surface generated by classes that can be represented by separating curves on the surface.
\begin{lemma} \label{apv0} (\cite{APV2017}, Lemma 4.2)  
 Let $S$ be a surface with only compact boundary components. Given a principal exhaustion $\{S_i\}$ of $S$ there exists a basis of $H_1^{sep}(S, \Z)$ comprised of curves in the boundary of the $S_i$.
\end{lemma}

\begin{lemma} \label{apv} (\cite{APV2017}, Proposition 3.3) 
Suppose $S$ is a surface with only compact boundary components. Then we have the following:
\end{lemma}

\begin{enumerate} 
    \item There is an injection $\phi$ from $H_1^{sep}(S, \Z)$ to $H^1(\PMCG(S), \Z)$, thought of as the group of all homomorphisms from $\PMCG(S)$ to $\Z$.
    \item Let $\alpha$ be a curve representing an element in $H_1^{sep}(S, \Z)$. The homomorphism $\phi(\alpha): \PMCG(S) \to \Z$ sends a handle shift $h$ to a nonzero element if and only if it cuts $\alpha$, and it sends any map in $\PMCGcc{S}$ to 0. We can assume $\phi(\alpha)$ sends a given handle shift cutting $\alpha$ to 1.
\end{enumerate}

\begin{proof} [Proof of Theorem \ref{structure}]

First assume $S$ has no planar ends or compact boundary components. The case of at most one end accumulated by genus was done in \cite{PV2018}, so assume $S$ has at least two ends accumulated by genus. Let $\{\alpha_i\}$ be a collection of curves forming a basis for $H_1^{sep}(S, \Z)$, which exists by Lemma \ref{apv0} and the fact that principal exhaustions always exist for surfaces with only compact boundary components. Now cut $S$ along each of the $\alpha_i$. Each separating curve in the cut surface bounds a compact surface, since otherwise the collection of curves above would not form a basis. Since any infinite type surface with more than one end has separating curves which do not bound a compact subsurface, it follows that each component of the cut surface is a Loch Ness Monster with $k \in \N \cup \{\infty\}$ compact boundary components added. Note this gives another proof of Lemma 4.11, and the collection of curves given by this lemma will provide an example of a basis for $H_1^{sep}(S, \Z)$.

Each component $Z$ of the cut surface can be modeled as $\R^2$ with $k$ open disks removed along the horizontal axis and handles attached periodically and vertically above each removed disk. Let $Y$ be the surface obtained from $[-1 , 1] \times [0, \infty) \subset \R^2$ by attaching a handle inside a small neighborhood about each interior integer point. We can properly embed $k$ disjoint copies of $Y$ into $Z$ so each copy of $[-1 , 1] \times \{0\} \subset Y$ is mapped to a different boundary component of $Z$. 

Now we paste all of the components back together to form the original surface $S$. We can choose the embeddings of $Y$ above so the union of their images is a collection of disjoint Strips with Genus. This then gives a collection of handle shifts $\{h_i\}$ where each $h_i$ cuts only $\alpha_i$. By Lemma \ref{apv}, we have homomorphisms $\phi(\alpha_i): \PMCG(S) \to \Z$ such that $\phi(\alpha_i)$ sends $h_i$ to 1 and every other $h_j$ to 0. Let $H$ be the subgroup topologically generated by the $\{h_i\}$. Since all of the $h_i$ commute, $H$ is a direct product of countably many copies of $\Z$. The product map $\phi = \Pi_{i=1}^n \phi(\alpha_i)$ gives a homomorphism from $\PMCG(S)$ to $H $. Then by Lemma \ref{apv}, we have a split exact sequence:

\begin{center}
\begin{tikzcd}
    1\arrow{r} & \overline{\PMCG_c(S)} \arrow{r} & \PMCG(S) \arrow{r}{\phi} & H \arrow{r}\arrow[bend left=33]{l}{s} & 1
\end{tikzcd}
\end{center}

where $s$ is inclusion. The cases of surfaces with planar ends and compact boundary components are done similarly. When there are planar ends, we choose handle shifts which miss the planar ends. Then we get the desired semidirect product.

\end{proof}

The general case is a corollary of this result using Lemma \ref{injective} along with a new version of the usual capping trick.

\begin{construction} \label{capping} (Capping Boundary Chains) Let $S$ be a surface with noncompact boundary components. 
 Using Lemma \ref{cutchains}, we can cut $S$ along curves so that the components of the cut surface each have at most one boundary chain. Let $\{S_i\}$ be the collection of components with exactly one boundary chain. By the final remarks in the proof of Lemma \ref{chaintodisk}, we can build each $S_i$ by adding topology to a disk with boundary points removed which we will call $D_i$. Now we cap the boundary chains of $S$ by attaching a copy of each $D_i$ to the boundary of $S_i \subseteq S$ via the identity. We will denote the resulting surface $\bar{S}$.
\end{construction}

As an example, capping the boundary chain of any Sliced Loch Ness Monster gives the Loch Ness Monster. Capping the boundary chain of a Strip with Genus gives the unique surface with empty boundary and exactly two ends both of which are accumulated by genus (often referred to as the Ladder Surface). This construction was chosen because the inclusion of a surface into the capped off surface induces a map on the ends spaces which preserves ends accumulated by genus and planar ends. Note there is a natural homomorphism 

\begin{equation} \label{eq:1}
i: \PMCG(S) \to \PMCG(\bar{S})
\end{equation}

induced by inclusion, and $i$ is injective by Lemma \ref{injective}.

\begin{theorem} \label{structureextended}
 Let $S$ be any infinite type surface. Then $$\PMCG(S) = \overline{\PMCG_c(S)} \rtimes H$$ where $\displaystyle H \cong \Z^{n-1}$ when there is a finite number $n > 1$ of ends of $S$ accumulated by genus, $\displaystyle H \cong \Z^\infty$ when there are infinitely many ends accumulated by genus, and $H$ trivial otherwise. Furthermore, $H$ is generated by pairwise commuting handle shifts.  
\end{theorem}

\begin{proof}
Recall the case of at most one end accumulated by genus was done in Theorem \ref{patelvlamisextended}. Assume $S$ is a surface with noncompact boundary components, without planar ends or compact boundary components, and with at least two ends accumulated by genus. Let $\bar{S}$ be the capped surface given by Construction \ref{capping}, and let $i$ be the homomorphism between pure mapping class groups from (\ref{eq:1}) above. Note $\bar{S}$ has the same number of ends accumulated by genus as $S$. By Theorem \ref{structure}, there is a split exact sequence as above with $\bar{S}$ in the place of $S$. Recall $H$ is the subgroup topologically generated by disjoint handle shifts $\{h_i\}$, and $s$ is the inclusion map. It suffices to show each of the $h_i$ can be chosen to be inside $i(\PMCG(S))$, because then by injectivity of $i$ we get a split exact sequence:
\begin{center}
\begin{tikzcd}
    1\arrow{r} & \overline{\PMCG_c(S)} \arrow{r} & \PMCG(S) \arrow{r}{\phi \circ i} & H \arrow{r}\arrow[bend left=33]{l}{i^{-1} \circ s} & 1
\end{tikzcd}
\end{center}

Apply Lemma \ref{cutchains} to cut $S$ along a collection of curves so each component of the cut surface has at most one boundary chain. As in Construction \ref{capping}, each of the components with boundary chains can be represented as disks with boundary points removed with additional topology added. In fact by the assumption that there are no planar ends, these components are Disks with Handles possibly with compact boundary components added. We can piece together compact exhaustions on the components to get an exhaustion $\{S_i\}$ for $S$, and using Proposition \ref{singlechainequiv} we can choose the exhaustion so $\partial S_i \setminus \partial S$ is always composed of separating curves and arcs with endpoints on boundary components of the same chain. Also we can assume the exhaustion satisfies the first condition in the definition of a principal exhaustion.

\begin{figure}[ht]
\begin{center}
\begin{tikzpicture}
\node[anchor=south west,inner sep=0] at (0,0) {\includegraphics[scale = 0.22]{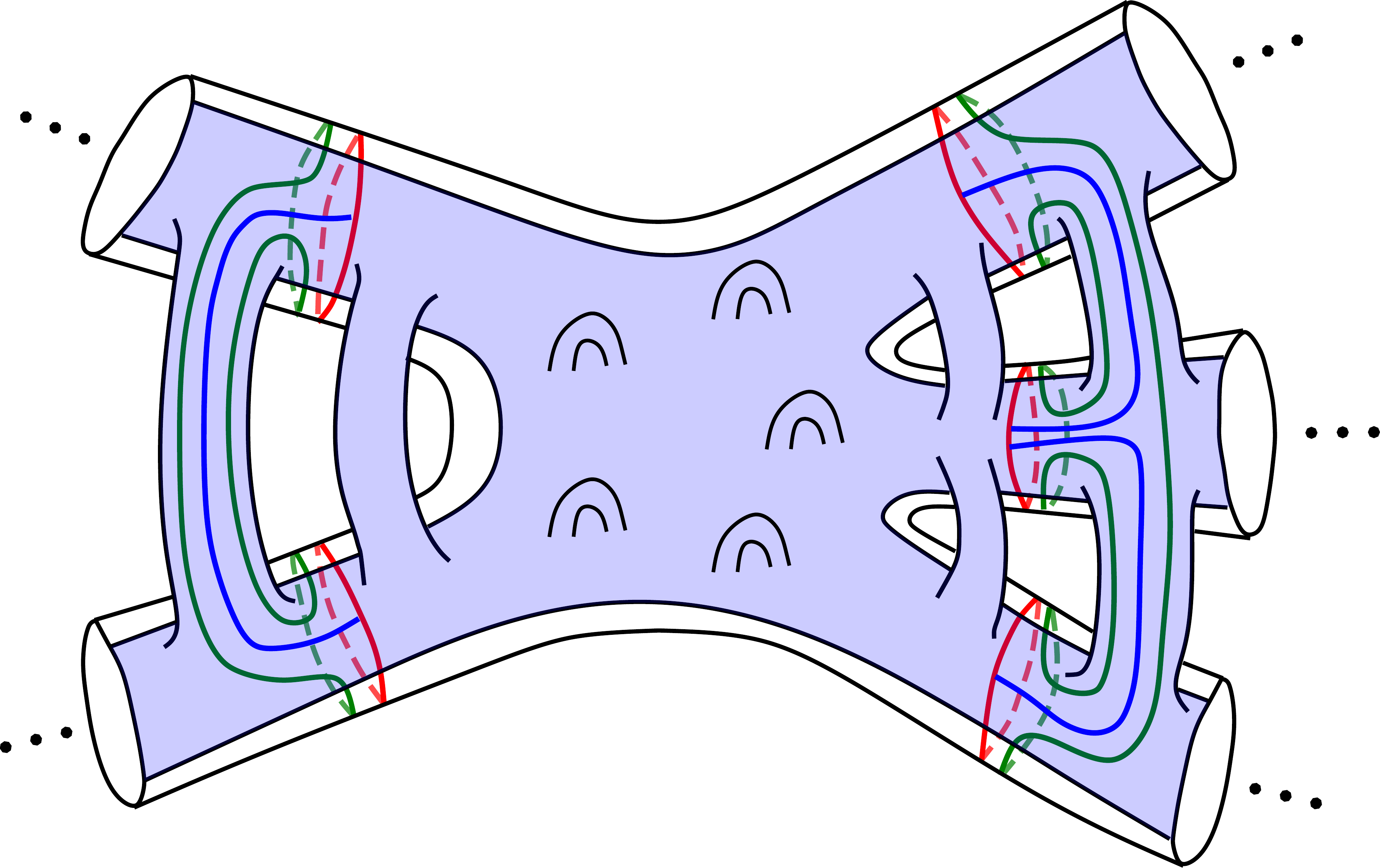}};

\end{tikzpicture}
\end{center}
\caption{A Disk with Handles shaded blue embedded in the capped off surface. The red curves are created by closing up arcs in the Disk with Handles. The blue arcs are used to replace the red curves with the green curves.}
\label{semi}
\end{figure}

Now we modify this exhaustion to get a principal exhaustion of $\bar{S}$. For every arc $\beta_k$ in $\partial S_i \setminus \partial S$, there is a corresponding arc $\beta^\prime_k$ in the attached disk which, together with $\beta_k$, closes up to a curve $\gamma_k$. The $\gamma_k$ together with the curves in $\partial S_i \setminus \partial S$ bound a compact subsurface $K_i \subset \bar{S}$. Then $\{K_i\}$ is a compact exhaustion for $\bar{S}$ which is not necessarily principal, but we can modify it so it becomes principal. Let $U$ be any complementary domain of $K_1$ such that $\partial U$ has $n > 1$ components. Connect each component of $\partial U$ together with $n-1$ disjoint arcs in $U \cap S$. Now enlarge $K_1$ by adding a closed regular neighborhood in $U$ of the arcs and the boundary components, then repeat this for each complementary domain with more than one boundary component. See Figure \ref{semi} for an example. Now remove some subsurfaces from the exhaustion so $K_1 \subset K_2$, and then repeat the above process for $K_2$. Continue in this manner to get a principal exhaustion.

Now we sketch the final details. Find a homology basis $\{\alpha_i\}$ of $H_1^{sep}(\bar{S}, \Z)$ comprised of curves that are boundary components for surfaces in the above principal exhaustion. Then we cut $\bar{S}$ along these curves, and we get components which are Loch Ness Monsters with compact boundary components added. Next we build the subgroup $H$ by taking the group topologically generated by disjoint handle shifts $h_i$ where each $h_i$ cuts $\alpha_i$ and no other curve in the basis. In this part of the proof there is a great deal of choice for how to embed these strips; in particular, we can assume the strips are contained in $S$. The remaining cases are done similarly to the proof of Theorem \ref{structure}.
\end{proof}

Now we show why Theorem \ref{structureextended} and Theorem A imply Theorem B. 

\begin{proof} [Proof of Theorem B]

The reverse directions of Theorem B are immediate from Theorem A. Now notice that the commutator subgroup of $$\PMCG(S) = \overline{\PMCG_c(S)} \rtimes H$$ is contained in $\overline{\PMCG_c(S)}$ since $H$ is abelian. Therefore, $\PMCG(S)$ cannot be perfect when $S$ has more than one end accumulated by genus. Since $\PMCG(S) = \overline{\PMCG_c(S)}$ when $S$ has one end accumulated by genus, we get the forward implications of Theorem B from the forward implications of Theorem A and the above remark. 
\end{proof}

\bibliography{bib}
\bibliographystyle{halpha}

\end{document}